\documentclass[11pt,a4paper]{article}
\usepackage[a4paper, top=2cm, bottom=2cm, left=2.5cm, right=2cm]{geometry}

\usepackage{newtxtext,newtxmath}

\usepackage{amsmath}
\usepackage{bm}

\usepackage{graphicx}
\usepackage{booktabs}
\usepackage{subcaption}

\usepackage{algorithm}
\usepackage{algpseudocode}

\usepackage{tikz}
\usetikzlibrary{arrows, positioning, shapes.geometric}

\usepackage[numbers]{natbib}

\usepackage{authblk}

\usepackage{enumitem}
\setlist{noitemsep, topsep=0pt, parsep=0pt, partopsep=0pt}

\usepackage{titlesec}
\titlespacing*{\section}{0pt}{2pt}{1pt}
\titlespacing*{\subsection}{0pt}{1pt}{0.5pt}
\titlespacing*{\subsubsection}{0pt}{0.5pt}{0.25pt}

\newtheorem{theorem}{Theorem}[section]
\newtheorem{lemma}[theorem]{Lemma}

\newtheorem{corollary}[theorem]{Corollary}
\newtheorem{definition}{Definition}[section]

\newtheorem{remark}{Remark}[section]

\newenvironment{proof}{\noindent\textbf{Proof.}\quad}{\hfill$\square$}

\newcommand{\Dalpha}{{}^C D_t^\alpha}

\usepackage[colorlinks=true, linkcolor=blue, citecolor=blue, urlcolor=blue]{hyperref}
\usepackage{authblk}

\title{Mathematical Analysis and Modeling of Ebola Virus Dynamics via Optimal Control and Neural Network Paradigms
 \thanks{A preliminary version of this work is available on arXiv: \url{https://arxiv.org/abs/2511.06303} or via DOI: \url{https://doi.org/10.48550/arXiv.2511.06303}.}}
\author[1,*]{Noor Muhammad}
\author[2,3]{Md. Nur Alam}
\author[1,*]{Zhang Shiqing}
\author[4]{Ali Akgül}%aakgul1@upvnet.upv.es
\affil[1]{School of Mathematics, Sichuan University, Chengdu, China, 610064}
\affil[2]{School of Mathematical Sciences, Sunway University, Bandar Sunway, Petaling Jaya 47500, Selangor Darul Ehsan, Malaysia}
\affil[3]{Department of Mathematics, Pabna University of Science \& Technology, Pabna-6600, Bangladesh}
\affil[4]{Siirt University, Art and Science Faculty, Department of Mathematics, 56100 Siirt, Turkey}
\affil[*]{Corresponding authors: \\
\textsuperscript{1}noormuhammad@stu.scu.edu.cn, noormustaffa681@gmail.com \\
\textsuperscript{1}zhangshiqing@scu.edu.cn}
\date{}
\begin{document}
\maketitle
\begin{abstract}
Ebola virus disease is a severe, often fatal infectious disease transmitted through direct contact with infected bodily fluids and contaminated surfaces. This study addresses the serious need for integrated predictive and control frameworks by developing a unified approach that connects intelligent outbreak prediction systems with mathematical models of Ebola transmission. We propose a novel nonlinear fractional-order eight-compartment SVEI$_s$I$_a$HDR model that integrates control measures, personal protection, vaccination, treatment, and safe burial to analyze their effectiveness in reducing Ebola transmission and improving outbreak prediction. Mathematical validation establishes the model's boundedness, non-negativity, and well-posedness. We derive the basic reproduction number R$_0$ via the next-generation matrix, identify disease-free and endemic equilibrium, and conduct stability analyses. Sensitivity analysis reveals R$_0$'s primary sensitivity to transmission rate $\beta$ (+1.00), incubation period $\sigma$ (+0.6250), and infectivity of the dead $\eta_d$ (+0.4678). An optimal control framework minimizes disease management costs (4 dollars per infection). Numerical simulation by the RK45 method demonstrates combined treatment and safe burial controls 86.5\% of target morbidity and mortality, with safe burial being the most effective single measure. We develop a disease-informed neural network (DINN) achieving 99.87\% prediction accuracy ($R^2=0.9987$) with minimal error (NMSE $10^{-4}$ to $10^{-5}$). The DINN accurately recovers 13 of 17 transmission parameters with $<9\%$ error, establishing a new interpretability benchmark. This integrated framework provides public health officials with a practical tool to test interventions and predict outbreaks through model-data integration.
\end{abstract}
\noindent\textbf{Keywords:} Ebola virus disease, fractional-order compartmental model, basic reproduction number (\(\mathcal{R}_0\)), optimal control, sensitivity analysis, stability analysis, outbreak prediction, disease-informed neural network (DINN).
\section{Introduction}
A serious infectious disease, Ebola Virus Disease (EVD) is transmitted to people by reservoirs of animals such as fruit bats. When the Ebola virus transmits from animal reservoirs into human populations, an outbreak occurs \cite{troncoso2015ebola, Condra2016}. The illness has historically significant case fatality rates, with outbreaks ranging from 25\% to 90\% \cite{who2014ebola}. It was first reported in the Democratic Republic of the Congo in 1976; the name Ebola was given due to the Ebola River \cite{ndanguza2013statistical}. With 28,646 confirmed infections and 11,323 deaths recorded, the most catastrophic epidemic was in West Africa (2014–2016), with almost a 50\% death ratio \cite{barboza2015challenges}. Fever, weakness, diarrhea, and, in extreme situations, hemorrhagic signs are common in infected people. There is a major issue in that these viruses' incubation periods differ, sometimes 2 days and sometimes 21 days; infected individuals cannot spread the virus throughout the incubation phase \cite{legrand2007understanding}. Fruit bats are considered to be the main natural source of infection in humans, and non-human primates serve as end hosts, maintaining cycles of transmission that can result in major outbreaks \cite{tulu2017fruit, jones2005fruit}. Controlling measures are complicated, particularly frequent contact paths that boost the spread of EVD. Being in contact with infectious body fluids, such as blood or excretions, results in subsequent human-to-human transmission. Contact with deceased individuals during traditional burial practices is a crucial and widely recognized route of transmission that has been a significant facilitator of prior epidemics \cite{Alimodeling2021, richards2018estimation}. In addition, the virus might continue to spread indirectly by remaining on contaminated surfaces. Due to the absence of apparent symptoms, various demographic categories, in particular people who participate in burial customs and asymptomatic individuals, have a possibility to spread infection throughout communities. Therefore, it is difficult for standard mathematical models to capture the EVD dynamics of transmission, as they have to immediately identify these various and context-specific threats \cite{fisman2014ebola, rachah2015modelling}. Ebola transmission is rendered much more difficult by the potential for reinfection. After recovery, immunity may not be total or long-lasting because protective antibodies may gradually decline in survivors \cite{wauquier2009human, macintyre2016ebola}. Individuals who have previously been infected can test positive for the virus years later, indicating potential re-exposure and infection, according to data from previous outbreaks \cite{lamunu2004containing, sobarzo2013persistent}. Additionally, research on populations that have been exposed frequently has discovered people who display no symptoms but do not have detectable antibodies, underscoring the heterogeneity in immune response \cite{heffernan2005perspectives, leroy2000human}. Now, we focus on a few equivalent mathematical models of EVD that have already been proposed by other researchers. Fractional calculus has been applied to infectious disease modeling in a number of research investigations, displaying more consistency with observed outbreak data than traditional approaches \cite{Altaf2019, yadav2023}. Fractional calculus has long shown the potential to model memory effects and other non-local connections and thus provides a well-intentioned framework of study to overcome some of the noted deficiencies \cite{podlubny1999fractional}.  Unlike integer-order derivatives, which always adopt Markovian processes, fractional operators capture the distributed suspensions that are characteristic of epidemiological systems, such as extended incubation periods, waning immunity, and environmental reservoirs of pathogens \cite{Amuneef2021}. Evidence of progress in the modeling of infectious diseases for which data are available and for which fractional calculus applies exhibits greater concord with experimental indication, especially for diseases with long-range chronological correlations and power law decline in the efficiency of involvements \cite{singh2020, AtifA2022}. The non-Markovian nature of some fractional derivatives offers a better-suited mathematical framework for modeling the transmission of disease in and among structured populations \cite{Baishya2021}.The SI model, such as \cite{berge2017simple,wang2015ebola}, the SEIR model \cite{webb2016model,siettos2015modeling}, the SEIRD model \cite{chowell2014transmission}, and the SEIRHD model \cite{zhang2022fractional,gomes2014assessing, siriprapaiwan2018generalized, adu2023fractal}, where $S$, $E$, $I$, $H$, $R$, and $D$ represent susceptible, exposed, infectious, hospitalized, recovered, and dead/deceased compartments, respectively \cite{imran2017modeling}, separated the Dead/Deceased ($D$) compartment into two compartments, buried deceased ($R_D$) and unburied dead ($F$). Similarly, in their SIRPD model, where $P$ represents environmental contamination, \cite{berge2017simple} included EVD transmission through contaminated surroundings of deceased patients. While models like the SEIHFDBR in \cite{Adu2024} represent notable developments, they remain limited by their integer-order formulation. Understanding how to effectively control the spread of infection and facilitate decision-making during an epidemic is one of the applications of mathematical modeling. For example, \cite{yusuf2012optimal} presented an SIR model-based optimum control issue with treatment and immunization as possible control strategies. They looked into the most effective vaccination and therapy combinations to lower the cost of those preventive measures. In \cite{lashari2012optimal} investigated various diseases transmitted by insects that directly infect the host population, such as dengue fever, malaria, and the West Nile virus, using optimal control analysis. In order to examine the transmission of Ebola, \cite{seck2022optimal} designed a model. They evaluated the effects of various preventative measures such as vaccination, public awareness campaigns, early identification campaigns, improved hospital sanitation, separation of people with infections, and geographical constraints. Another field of research is the integration of machine learning with traditional epidemiological models. When handling noisy or incomplete surveillance information, physics-informed neural networks (PINNs) demonstrate particular potential for parameter estimation and prediction \cite{raissi2019physics, willard2020integrating, TangY2020}. A comprehensive analysis of the last five years' worth of literature, however, reveals enduring shortcomings and disparate study directions. Although fractional models have been created, they are frequently not integrated with intricate multi-pathway transmission designs specific to EVD, for instance. Likewise, machine learning applications often function independently of mathematical models, which limits their understanding and generalization, and optimal control formulations usually use simplified models that do not fully capture the complexity of transmission dynamics \cite{sarumi2020}. Nevertheless, despite these developments, there are still important research gaps in present literature. Firstly, there is a gap between the full representation of all recognized EVD transmission channels (symptomatic, asymptomatic, and burial customs) and fractional-order models, which account for memory effects. There is a methodological gap given that detailed transmission models generally employ integer-order approaches, whereas the majority of fractional models simplify transmission dynamics. Second, the application of optimal control theory to high-dimensional fractional systems that reflect EVD dynamics has been insufficient. One major challenge that has limited practical applicability is the mathematical complexity of determining optimal control policies for such systems, especially when including numerous intervention strategies with different costs and effectiveness. \cite{sharmi2017, kochanczyk2021}. Third, the majority of applications currently only combine biological process-based models with data-driven machine learning techniques. Hybrid frameworks have been proposed, but they typically lack the seamless combination needed for real-time parameter estimation and prediction during current outbreaks \cite{bertozzi2020modeling}. This study develops a novel computational framework with the particular objectives that follow in order to address these connected challenges. The practical objective is to provide public health decision-makers an efficient computational tool that can forecast disease paths with measured uncertainty, optimize allocation of resources for maximum impact, and simulate outbreak scenarios under different intervention strategies. The following is a list of this study's primary innovations and contributions:
\begin{itemize}[leftmargin=*]
\item We developed a novel nonlinear fractional-order eight-compartment mathematical model with memory effects and multiple control intervention evaluation of Ebola transmission dynamics. To prove the validity of the model, we presented boundedness, positivity, and well-posedness. We obtained the basic reproduction number$R_0$ by the next -generation matrix approach and the disease-free and endemic equilibrium points and then conducted thorough stability analysis of both states. Through sensitivity analysis, we determined that parameters such as the transmission rate, the incubation period, and the infectivity of the deceased had the greatest effects on $R_0$.
\item We established the optimal control model that would help to reduce the cost of managing the disease by considering four intervention measures that included personal protection, vaccination, treatment and safe burial. Our theoretical analysis by applying Pontryagin's Maximum Principle and numerical simulation using the RK45 method permitted identifying the cost-effective strategies, and the level of their effectiveness in terms of transmission and mortality reduction.
\item Motivated from physics-informed neural networks, we designed a new framework for disease-informed neural networks (DINN).  This computational approach achieves high prediction accuracy across all disease stages and demonstrates strong parameter recovery capability, providing an interpretable tool for outbreak forecasting and intervention assessment.
\end{itemize}
This work significantly improves both academic research and healthcare practice by successfully addressing the existing gap between theoretical modeling and actual application. It offers investigators a new methodological framework that shows how machine learning, optimal control theory, and fractional calculus may be successfully used for the scientific study of infectious diseases. 
This paper proceeds as follows. In Section 2, I discuss the elementary ideas of fractional calculus, which is key to developing the model. In Section 3, I develop an eight-compartment, fractional-order Ebola transmission model. Section 4 presents the complete mathematical analysis of the model, including proofs of existence and uniqueness of solutions, identification of the invariant region and attractivity, derivation of the basic reproduction number ($\mathcal{R}_0$), and both local and global stability analyses for the disease-free and endemic equilibrium points. Section 5 performs a global sensitivity analysis that attempts to quantity and prioritize the impact of key model parameters on the disease transmission dynamics. Section 6 establishes and solves optimal control problem through the Maximum Principle of Pontryagin founding the most effective intervention strategies.  In section 7, the formation of a disease-informed neural network (DINN) model along with the relevant numerical data is developed and fully analyzed. Lastly, in Section 8 the main findings are synthesized, the methodological limitations are critically discussed, and future research directions outlined.
\section{Preliminaries}
\label{sec:preliminaries}
This part sets out the basic principles of fractional calculus which is essential to the construction and analysis of our model.
\begin{definition}
\label{def:caputo}
Let $\alpha \in (0,1]$ and $u \in C^{1}([0,T])$. For any function $u$ which is continuously differentiable in the closure $[0,T]$ the Caputo fractional derivative of order $\alpha$ is defined in \cite{podlubny1999fractional} as follows.
\[
{}^{C}D_{t}^{\alpha} u(t) = \frac{1}{\Gamma(1-\alpha)} \int_{0}^{t} (t-\tau)^{-\alpha} u'(\tau)  d\tau.
\]
This operator is preferred for modeling dynamical systems as it accommodates standard initial conditions, and ${}^{C}D_{t}^{\alpha} u(t) \to u'(t)$ as $\alpha \to 1^{-}$, recovering the classical integer-order derivative.
\end{definition}
\begin{definition}
\label{def:fractional-stability}
The equilibrium point $y^*$ of the fractional-order system ${}^{C}D_{t}^{\alpha} y(t) = f(t, y(t))$ is said to be \emph{Mittag-Leffler stable} if
\[
\| y(t) - y^* \| \leq \left[m(y(0)-y^*) E_{\alpha, 1}(-\lambda t^{\alpha})\right]^{b}.
\]
where $\alpha \in (0,1)$, $\lambda \geq 0$, $b > 0$, $m(0)=0$, and $m(y) \geq 0$. Mittag-Leffler stability implies asymptotic stability \cite{li2010stability}.
\end{definition}
\begin{lemma}
\label{thm:stability-condition}
Let $\alpha \in (0,1]$. An equilibrium point $y^*$ of the autonomous system ${}^{C}D_{t}^{\alpha} y(t) = f(y(t))$ is asymptotically stable if all eigenvalues $\lambda_i$ of the Jacobian matrix $J = \partial f / \partial y |_{y^*}$ satisfy the condition \cite{li2010stability}.
\[
|\arg(\lambda_i)| > \frac{\alpha \pi}{2}, \quad \forall i.
\]
This condition generalizes the classical Routh-Hurwitz criterion to fractional-order systems.
\end{lemma}
\section{Model formulation}
We formulate a Caputo fractional-order model ($\alpha \in (0,1]$) that incorporates the multi-pathway transmission, memory effects, and time-varying interventions of Ebola. The governing system is given as follows.
\begin{equation}
\begin{aligned}
\Dalpha S(t) &= \Lambda - \frac{c\beta S (I_s + \eta_a I_a + \eta_d D)}{N} + \omega V - (v+\mu)S, \\
\Dalpha V(t) &= vS - (1-\varepsilon)\frac{c\beta V (I_s + \eta_a I_a + \eta_d D)}{N} - (\mu+\omega) V, \\
\Dalpha E(t) &= \frac{c\beta [S + (1-\varepsilon)V] (I_s + \eta_a I_a + \eta_d D)}{N} - (\mu+\sigma)E, \\
\Dalpha I_s(t) &= p\sigma E - (\gamma_s + \delta_s + h_s + \mu)I_s, \\
\Dalpha I_a(t) &= (1-p)\sigma E - (\gamma_a + \mu)I_a, \\
\Dalpha H(t) &= h_s I_s - (\gamma_h + \delta_h +\mu )H, \\
\Dalpha D(t) &= \delta_s I_s + \delta_h H - \mu_d D, \\
\Dalpha R(t) &= \gamma_s I_s + \gamma_a I_a + \gamma_h H - \mu R.
\end{aligned}
\label{eq:ebola_model}
\end{equation}
The force of infection is given by
\[
\lambda = \frac{c\beta (I_s + \eta_a I_a + \eta_d D)}{N},
\]
and for notational simplicity, we define the model as
\begin{align*}
q_0 &= ( v+\mu), & q_1 &= (\mu+\omega), & q_2 &= (\mu+\sigma), \\
q_3 &= (\gamma_s + \delta_s + h_s + \mu), & q_4 &= (\gamma_a + \mu), & q_5 &= (\gamma_h + \delta_h + \mu), \\
q_6 &= \mu_d.
\end{align*}
The model parameters include natural mortality ($\mu$) and recruitment ($\Lambda$), vaccination rate $v(t)$ with waning immunity ($\omega$), and force of infection $\lambda(t)$ driving exposures. Exposed individuals ($E$) progress to infectious states at rate $\sigma$, splitting into symptomatic ($I_s$, proportion $p$) and asymptomatic ($I_a$, $1-p$) pathways. Symptomatic cases face hospitalization ($h_s$), disease-induced mortality ($\delta_s$), or recovery ($\gamma_s$). Asymptomatic and hospitalized individuals recover at rates $\gamma_a$ and $\gamma_h$, respectively, with hospital mortality ($\delta_h$). Deceased individuals ($D$) transmit infection until safe burial ($\mu_d$). Transmission coefficients $\beta$, $\eta_a$, and $\eta_d$ govern infectivity from symptomatic, asymptomatic, and post-mortem sources.The system is governed by,
\begin{equation}
\begin{aligned}
\Dalpha S(t) &= \Lambda - \lambda S + \omega V - q_0 S, \\
\Dalpha V(t) &= vS - (1-\varepsilon)\lambda V - q_1 V, \\
\Dalpha E(t) &= \lambda [S + (1-\varepsilon)V] - q_2 E, \\
\Dalpha I_s(t) &= p\sigma E - q_3 I_s, \\
\Dalpha I_a(t) &= (1-p)\sigma E - q_4 I_a, \\
\Dalpha H(t) &= h_s I_s - q_5 H, \\
\Dalpha D(t) &= \delta_s I_s + \delta_h H - \mu_d D, \\
\Dalpha R(t) &= \gamma_s I_s + \gamma_a I_a + \gamma_h H - \mu R.
\end{aligned}
\label{eq:ebola_system}
\end{equation}
The fractional-order system \eqref{eq:ebola_system} is analyzed with the following initial conditions:
\[
S(0) = S_0 \geq 0, \quad V(0) = V_0 \geq 0, \quad E(0) = E_0 \geq 0, \quad I_s(0) = I_{s_0} \geq 0,
\]
\[
I_a(0) = I_{a_0} \geq 0, \quad H(0) = H_0 \geq 0, \quad D(0) = D_0 \geq 0, \quad R(0) = R_0 \geq 0.
\]
In this Caputo fractional formulation, the system in \eqref{eq:ebola_system} captures the dynamics of Ebola transmission in an eight-dimensional frame. This construction involves ‘epidemic’ components such as susceptible ($S$), vaccinated ($V$), exposed ($E$), symptomatically ($I_s$) and asymptomatically ($I_a$) infectious, hospitalized ($H$), deceased ($D$), and recovery ($R$) stages, thus integrating the critical structure of underlying dynamics. Furthermore, the model participates in the fractional-order system of differential equations, thus inserting memory effects.
\begin{table}[!ht]
\centering
\caption{Parameters for the fractional-order Ebola model.}
\label{tab:model_parameters}
\footnotesize
\begin{tabular}{|p{1.4cm}|p{4.2cm}|c|c|c|}
\hline
\textbf{Symbol} & \textbf{Description} & \textbf{Value} & \textbf{Range} & \textbf{Source} \\
\hline
$\beta$   & Transmission rate & 0.287 & 0.2--0.4 & \cite{Xia2015} \\
\hline
$\eta_d$  & Deceased infectiousness & 0.734 & 0.5--0.9 & \cite{Xia2015} \\
\hline
$\eta_a$  & Asymptomatic infectiousness & 0.523 & 0.3--0.7 & Assumed (Fitted) \\
\hline
$\sigma$  & Incubation rate (day$^{-1}$) & 0.094 & 0.05--0.15 & \cite{Goeijenbeir2014} \\
\hline
$p$       & Symptomatic proportion & 0.712 & 0.5--0.8 & \cite{Goeijenbeir2014} \\
\hline
$\gamma_s$& Symptomatic recovery (day$^{-1}$) & 0.068 & 0.05--0.1 & \cite{berge2017simple} \\
\hline
$\gamma_a$& Asymptomatic recovery (day$^{-1}$) & 0.089 & 0.07--0.12 & \cite{berge2017simple} \\
\hline
$\delta_s$& Symptomatic death (day$^{-1}$) & 0.103 & 0.05--0.15 & \cite{Chowell2014} \\
\hline
$\delta_h$& Hospitalized death (day$^{-1}$) & 0.067 & 0.03--0.09 & \cite{Chowell2014} \\
\hline
$h_s$     & Hospitalization rate & 0.312 & 0.2--0.4 & \cite{Lau2017} \\
\hline
$\Lambda$ & Recruitment rate & 100 & 100--1000 & \cite{Berg2017} \\
\hline
$\mu$     & Natural mortality (day$^{-1}$) & $3.5\times10^{-5}$ & -- & \cite{Berg2017} \\
\hline
$v(t)$    & Vaccination rate & -- & 0.005--0.08 & Assumed (Control) \\
\hline
$\varepsilon$ & Vaccine efficacy & -- & 0.85--0.95 & Assumed (Control) \\
\hline
$\omega$  & Waning immunity (day$^{-1}$) & -- & 0.0027--0.0037 & Assumed \\
\hline
$\alpha$  & Fractional order & 0.85 & 0.75--0.95 & Assumed (Model) \\
\hline
\end{tabular}
\end{table}
\begin{figure}[!ht]
\centering
\begin{tikzpicture}[
     mycircle/.style={circle, draw=black, thick, minimum size=1.8cm, fill opacity=1., 
                    text=black, font=\bfseries\itshape\LARGE}, 
    arrow/.style={->, >=stealth, thick, shorten >=2pt, shorten <=2pt},
    dashedarrow/.style={red, dashed, thick, ->, >=stealth, shorten >=2pt, shorten <=2pt},
    label/.style={font=\small, midway}
]
\node[mycircle, fill=green!100] (S) at (0,5) {$\mathbb{S}$};
\node[mycircle, fill=yellow!100] (V) at (-1.5,-1.2) {$\mathbb{V}$};
\node[mycircle, fill=orange!100] (E) at (5,4) {$\mathbb{E}$};
\node[mycircle, fill=red!80] (Is) at (10,5) {$\mathbb{I}_s$};
\node[mycircle, fill=red!80] (Ia) at (2,-5) {$\mathbb{I}_a$};
\node[mycircle, fill=blue!60] (H) at (12,0) {$\mathbb{H}$};
\node[mycircle, fill=gray!95] (D) at (9.5,-5) {$\mathbb{D}$};
\node[mycircle, fill=green!100] (R) at (5,0) {$\mathbb{R}$};
% Recruitment
\draw[<-] (S) -- ++(-1.5,0) node[left, font=\small] {$\Lambda$};
% Transmission arrows
\draw[dashedarrow] (S) to[bend left=25] node[label, below] {$\lambda S$} (E);
\draw[dashedarrow] (S) to[bend left=15] node[label, right, pos=.4] {$vS$} (V);
\draw[dashedarrow] (V) to[bend right=10] node[label, left, pos=0.8] {$(1-\varepsilon)\lambda V$} (E);
\draw[dashedarrow] (V) to[bend left=25] node[label, right, pos=0.3] {$\omega V$} (S);
% Progression arrows
\draw[dashedarrow] (E)  to[bend left=25] node[label, below, pos=0.7] {$p\sigma E$} (Is);
\draw[dashedarrow] (E) to[bend right=15] node[label, right, pos=0.2] {$(1-p)\sigma E$} (Ia);
% Recovery arrows
\draw[dashedarrow] (Is) to[bend right=10] node[label, right, pos=0.7] {$\gamma_s I_s$} (R);
\draw[dashedarrow] (Ia) to[bend right=10] node[label, right, pos=0.7] {$\gamma_a I_a$} (R);
\draw[dashedarrow] (H) to[bend right=10] node[label, above, pos=0.7] {$\gamma_h H$} (R);
% Hospitalization and death arrows
\draw[dashedarrow] (Is)  to[bend left=25] node[label, left, pos=0.3] {$h_s I_s$} (H);
\draw[dashedarrow] (Is) to[bend left=20] node[label, right, pos=0.7] {$\delta_s I_s$} (D);
\draw[dashedarrow] (H)  to[bend left=25] node[label, left, pos=0.7] {$\delta_h H$} (D);
\draw[dashedarrow] (Ia) to[bend left=15] node[label, above, pos=0.8] {$\delta_a I_a$} (D);
\draw[dashedarrow] (Ia) -- node[label, left, pos=0.8] {$h_a I_a$} (H);
% Natural mortality arrows (dashed)
\draw[dashedarrow] (S) -- ++(1.5,-1.5) node[below, font=\small] {$\mu S$};
\draw[dashedarrow] (V) -- ++(1.5,0) node[below, font=\small] {$\mu V$};
\draw[dashedarrow] (E) to[bend left=5] ++(1,3) node[above, font=\small] {$\mu E$};
\draw[dashedarrow] (Is) -- ++(-1,-1.5) node[below, font=\small] {$\mu I_s$};
\draw[dashedarrow] (Ia) to[bend left=20] ++(-2.7,2.3) node[above, font=\small] {$\mu I_a$};
\draw[dashedarrow] (H) -- ++(-0.5,1.5) node[above, font=\small] {$\mu H$};
\draw[dashedarrow] (D) -- ++(0,1.5) node[above, font=\small] {$\mu_d D$};
\draw[dashedarrow] (R) -- ++(1.5,-1.5) node[below, font=\small] {$\mu R$};
\end{tikzpicture}
\caption{Transmission architecture of the fractional Ebola model.}
\label{fig:ebola_diagram}
\end{figure}
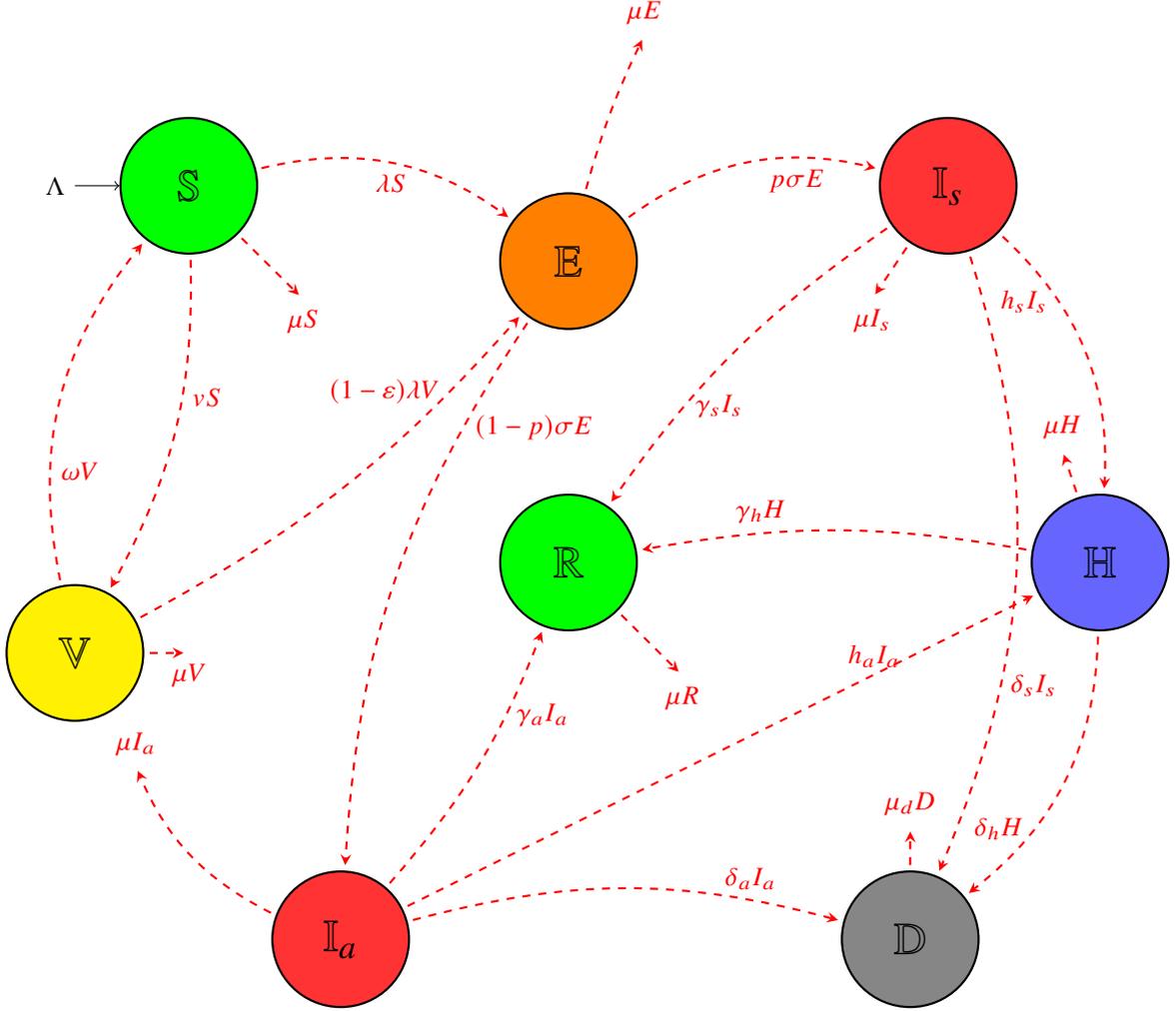
\section{Mathematical Analysis}
\subsection{Well-Posedness of the Model}
We establish the mathematical well-posedness of the fractional-order Ebola system through existence and uniqueness analysis.
\begin{theorem}
Given any $\alpha \in (0,1]$, and initial conditions ae $\mathbf{X}_0 = (S_0, V_0, E_0, I_{s0}, I_{a0}, H_0, D_0, R_0) \in \mathbb{R}_+^8$, the fractional-order model \eqref{eq:ebola_system} has the unique solution $\mathbf{X}(t) = (S(t), V(t), E(t), I_s(t), I_a(t), H(t), D(t), R(t))$ for all $t \geq 0$.
\end{theorem}
\begin{proof}
Define the state vector $\mathbf{X}(t) = (S, V, E, I_s, I_a, H, D, R)^T$ and express the system as
\[
\Dalpha \mathbf{X}(t) = \mathbf{F}(t, \mathbf{X}(t)), \quad \mathbf{X}(0) = \mathbf{X}_0.
\]
The equivalent Volterra integral formulation is
\begin{equation}
\mathbf{X}(t) = \mathbf{X}_0 + \frac{1}{\Gamma(\alpha)} \int_0^t (t-s)^{\alpha-1} \mathbf{F}(s, \mathbf{X}(s)) ds, \label{eq:volterra}
\end{equation}
Assume the operator $\mathcal{T}: C([0,T], \mathbb{R}^8) \to C([0,T], \mathbb{R}^8)$ defined by
\[
(\mathcal{T}\mathbf{X})(t) = \mathbf{X}_0 + \frac{1}{\Gamma(\alpha)} \int_0^t (t-s)^{\alpha-1} \mathbf{F}(s, \mathbf{X}(s))ds,
\]
For $\mathbf{X}, \mathbf{Y} \in C([0,T], \mathbb{R}^8)$ with $\|\mathbf{X}\| = \sup_{t \in [0,T]} \|\mathbf{X}(t)\|_\infty$, the Lipschitz continuity of $\mathbf{F}$ follows from bounded parameters and finite state variables over compact intervals. Specifically, there exists $L > 0$ such that
\begin{equation}
\|\mathbf{F}(t,\mathbf{X}) - \mathbf{F}(t,\mathbf{Y})\|_\infty \leq L \|\mathbf{X} - \mathbf{Y}\|_\infty,
\label{eq:lipschitz_condition}
\end{equation}
The contraction estimate provides that
\begin{equation}
\|\mathcal{T}\mathbf{X} - \mathcal{T}\mathbf{Y}\| \leq \frac{L \tau^\alpha}{\Gamma(\alpha+1)} \|\mathbf{X} - \mathbf{Y}\|.
\end{equation}
Selecting $\tau > 0$ such that $\frac{L \tau^\alpha}{\Gamma(\alpha+1)} < 1$ ensures $\mathcal{T}$ is a contraction on $C([0,\tau], \mathbb{R}^8)$. Banach's fixed point theorem guarantees a unique local solution $\mathbf{X}^* \in C([0,\tau], \mathbb{R}^8)$. Global extension to $[0,\infty)$ follows from \eqref{eq:lipschitz_condition} the linear growth condition$\|\mathbf{F}(t,\mathbf{X})\|_\infty \leq M(1 + \|\mathbf{X}\|_\infty)$ for some $M > 0$, this guarantees that solutions are always bounded. It is demonstrated that the solution exists and remains unique globally using an iterative continuation technique.
\end{proof}
\subsection{Invariant Region and Attractivity}
The fractional-order Ebola model \eqref{eq:ebola_system} evolves within the biologically feasible region is
\begin{equation} \label{eq:feasible_region}
\Omega = \left\{ (S, V, E, I_s, I_a, H, D, R) \in \mathbb{R}^8_{\geq 0} :\quad D \leq \frac{(\delta_s + \delta_h)}{\mu_d}\left(\frac{\Lambda}{\mu} + \epsilon\right), N_L \leq (\frac{\Lambda}{\mu} + \epsilon) \right\},
\end{equation}
for $\epsilon > 0$, where $N_L = S + V + E + I_s + I_a + H + R$ denotes the living population.
\begin{theorem}
The region $\Omega$ is defined as
\begin{enumerate}
    \item[(i)] Positively invariant under model \eqref{eq:ebola_system},
    \item[(ii)] Globally attractive in $\mathbb{R}^{8}_{\geq 0}$.
\end{enumerate}
\end{theorem}
\begin{proof}
Applying the fractional positivity lemma from \cite{podlubny1999fractional}, we verify boundary behavior as,
\begin{equation}
\begin{aligned}
&\Dalpha S|_{S=0} = \Lambda + \omega V \geq 0, \quad
\Dalpha V|_{V=0} = vS \geq 0, \\
&\Dalpha E|_{E=0} = \frac{c\beta[S + (1-\varepsilon)V](I_s + \eta_a I_a + \eta_d D)}{N} \geq 0, \\
&\Dalpha I_s|_{I_s=0} = p\sigma E \geq 0, \quad
\Dalpha I_a|_{I_a=0} = (1-p)\sigma E \geq 0, \\
&\Dalpha H|_{H=0} = h_s I_s \geq 0, \quad
\Dalpha D|_{D=0} = \delta_s I_s + \delta_h H \geq 0, \\
&\Dalpha R|_{R=0} = \gamma_s I_s + \gamma_a I_a + \gamma_h H \geq 0.
\end{aligned}
\label{eq:non_negativity_conditions}
\end{equation}
Non-negativity of all boundary fluxes ensures $\mathbb{R}_+^8$ is invariant.

The total living population, defined as $N_L(t) = S(t) + V(t) + E(t) + I_s(t) + I_a(t) + H(t) + R(t)$, evolves according to
\begin{equation}
\Dalpha N_L = \Lambda - \mu N_L - (\delta_s I_s + \delta_h H) \leq \Lambda - \mu N_L,
\label{eq:N_L_dynamics}
\end{equation}
Applying the Laplace transform to the inequality $\Dalpha N_L \leq \Lambda - \mu N_L$ gives
\begin{equation}
N_L(s) \leq \frac{\Lambda}{s(s^\alpha + \mu)} + N_L(0) \frac{s^{\alpha-1}}{s^\alpha + \mu},
\label{eq:boundedness_laplace}
\end{equation}
Applying the inverse Laplace transform yields
\[
N_L(t) \leq N_L(0) E_{\alpha,1}(-\mu t^\alpha) + \Lambda t^\alpha E_{\alpha,\alpha+1}(-\mu t^\alpha).
\]
Taking the limit as \(t \to \infty\), we use the asymptotic behaviors \(E_{\alpha,1}(-\mu t^\alpha) \to 0\) and \(t^\alpha E_{\alpha,\alpha+1}(-\mu t^\alpha) \to 1/\mu\),

which gives the bound
\[
\limsup_{t \to \infty} N_L(t) \leq \frac{\Lambda}{\mu}.
\]

For deceased individuals, we derive from the differential inequality
\[
\Dalpha D \leq (\delta_s + \delta_h)\left(\frac{\Lambda}{\mu} + \epsilon\right) - \mu_d D,
\]

and consequently obtain
\[
\limsup_{t \to \infty} D(t) \leq \frac{\delta_s + \delta_h}{\mu_d}\left(\frac{\Lambda}{\mu} + \epsilon\right).
\]
Thus, all trajectories are eventually absorbed into $\Omega$.
\end{proof}
\subsection{Ebola-Free Equilibrium}
The state in which the population is completely free of Ebola virus infection is known as the disease-free equilibrium (DFE). At this equilibrium, all infected compartments vanish, meaning
\[
E = I_s = I_a = H = D = 0.
\]
The system \eqref{eq:ebola_system} then reduces to
\begin{align*}
0 &= \Lambda - q_0 S + \omega V, \\
0 &= v S - q_1 V, \\
0 &= -\mu R.
\end{align*}
From the third equation, we obtain $R^* = 0$. The first two equations form the linear system
\begin{align*}
q_0 S - \omega V &= \Lambda, \\
-v S + q_1 V &= 0.
\end{align*}
The coefficient matrix
\[
A = \begin{pmatrix}
q_0 & -\omega \\
-v & q_1
\end{pmatrix}
\]
has determinant
\[
\det(A) = q_0 q_1 - \omega v = \mu(q_1 + v) > 0,
\]
which guarantees a unique solution. Applying Cramer's rule yields
\[
S^* = \frac{\Lambda q_1}{\mu(q_1 + v)}, \quad V^* = \frac{\Lambda v}{\mu(q_1 + v)}.
\]

Thus, the disease-free equilibrium is given by
\begin{equation}
\mathcal{E}_0 = \left(\frac{\Lambda q_1}{\mu(q_1 + v)}, \frac{\Lambda v}{\mu(q_1 + v)}, 0, 0, 0, 0, 0, 0\right).
\label{eq:DFE}
\end{equation}
This biologically consistent DFE serves as the basis for stability analysis and the derivation of the basic reproduction number.
\subsection{Basic Reproduction Number \texorpdfstring{$\mathcal{R}_0$}{R0}}
The basic reproduction number $\mathcal{R}_0$ is computed using the next-generation matrix method \cite{van2002reproduction}. Considering the infected subsystem $X = (E, I_s, I_a, H, D)^T$, we define the new infection matrix $F$ and the transition matrix $V$ as
\[
F = \begin{pmatrix}
0 & \beta\Psi & \beta\eta_a\Psi & 0 & \beta\eta_d\Psi \\
0 & 0 & 0 & 0 & 0 \\
0 & 0 & 0 & 0 & 0 \\
0 & 0 & 0 & 0 & 0 \\
0 & 0 & 0 & 0 & 0
\end{pmatrix}, \qquad
V = \begin{pmatrix}
q_2 & 0 & 0 & 0 & 0 \\
-p\sigma & q_3 & 0 & 0 & 0 \\
-(1-p)\sigma & 0 & q_4 & 0 & 0 \\
0 & -h_s & 0 & q_5 & 0 \\
0 & -\delta_s & 0 & -\delta_h & q_6
\end{pmatrix},
\]
where $\Psi = \frac{S^* + (1-\varepsilon)V^*}{N^*}$ and $N^* = \frac{\Lambda}{\mu}$.
The inverse of the transition matrix is
\[
V^{-1} = \begin{pmatrix}
\frac{1}{q_2} & 0 & 0 & 0 & 0 \\
\frac{p\sigma}{q_2 q_3} & \frac{1}{q_3} & 0 & 0 & 0 \\
\frac{(1-p)\sigma}{q_2 q_4} & 0 & \frac{1}{q_4} & 0 & 0 \\
\frac{p\sigma h_s}{q_2 q_3 q_5} & \frac{h_s}{q_3 q_5} & 0 & \frac{1}{q_5} & 0 \\
\frac{p\sigma(\delta_s q_5 + \delta_h h_s)}{q_2 q_3 q_5 q_6} & \frac{\delta_s q_5 + \delta_h h_s}{q_3 q_5 q_6} & 0 & \frac{\delta_h}{q_5 q_6} & \frac{1}{q_6}
\end{pmatrix}.
\]
The next-generation matrix is $K = FV^{-1}$, and its spectral radius gives the basic reproduction number
\begin{equation}
\mathcal{R}_0 = \beta\Psi\left[\frac{p\sigma}{q_2}\left(\frac{1}{q_3} + \frac{\eta_d(\delta_s q_5 + \delta_h h_s)}{q_3 q_5 q_6}\right) + \frac{\eta_a(1-p)\sigma}{q_2 q_4}\right].
\label{eq:R0_general}
\end{equation}

Substituting the expression for $\Psi$ yields the explicit form
\begin{equation}
\mathcal{R}_0 = \frac{\beta\sigma[S^* + (1-\varepsilon)V^*]}{(\Lambda/\mu) q_2}\left[\frac{p}{q_3}\left(1 + \frac{\eta_d(\delta_s q_5 + \delta_h h_s)}{q_5 q_6}\right) + \frac{\eta_a(1-p)}{q_4}\right].
\label{eq:R0_explicit}
\end{equation}
This threshold quantity governs the disease invasion and extinction dynamics.
\begin{figure}[!ht]
\centering
\begin{subfigure}{0.45\textwidth}
\centering
\includegraphics[width=\textwidth]{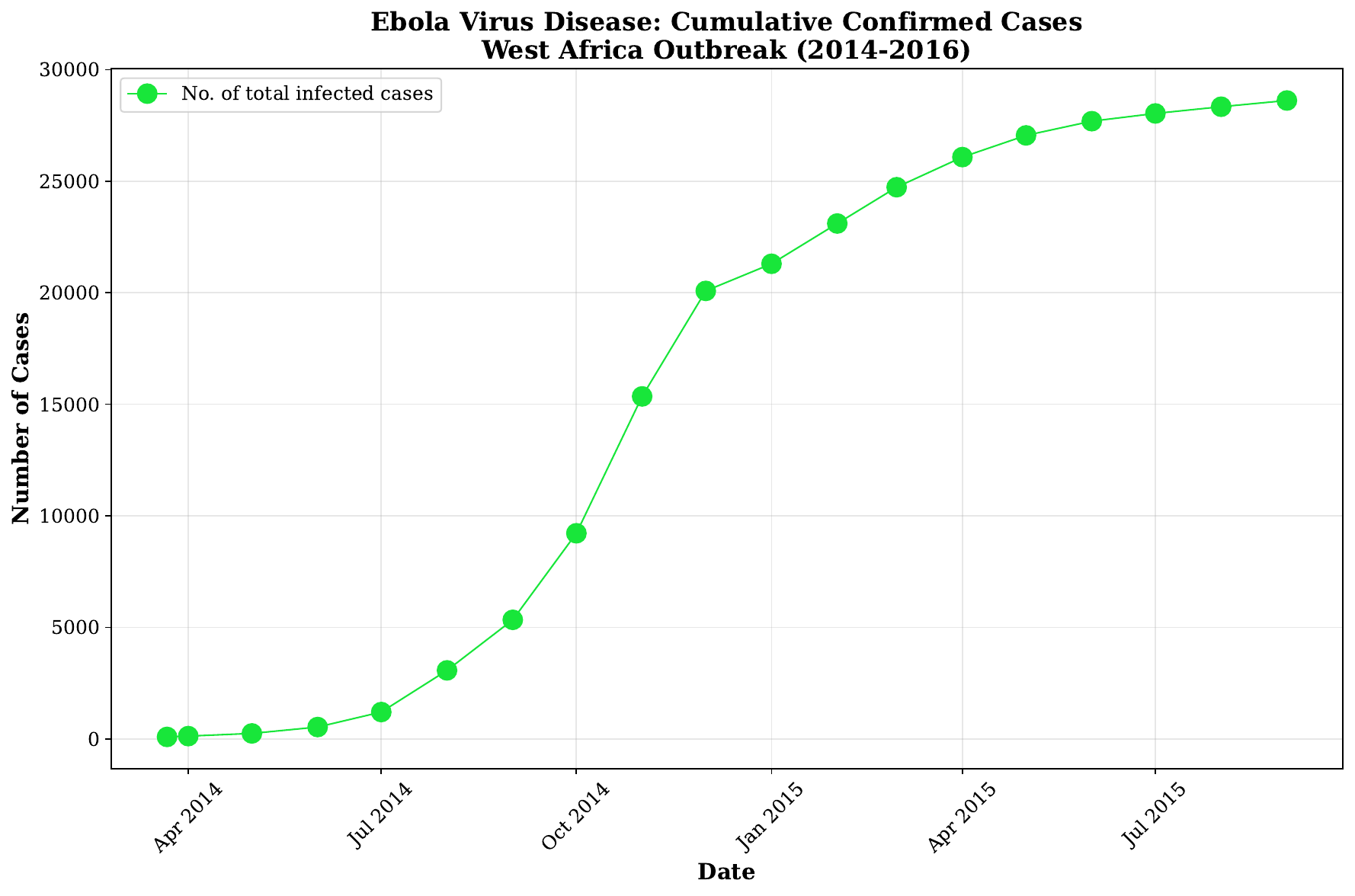}
\caption{Cumulative confirmed cases progression}
\end{subfigure}
\hfill
\begin{subfigure}{0.45\textwidth}
\centering
\includegraphics[width=\textwidth]{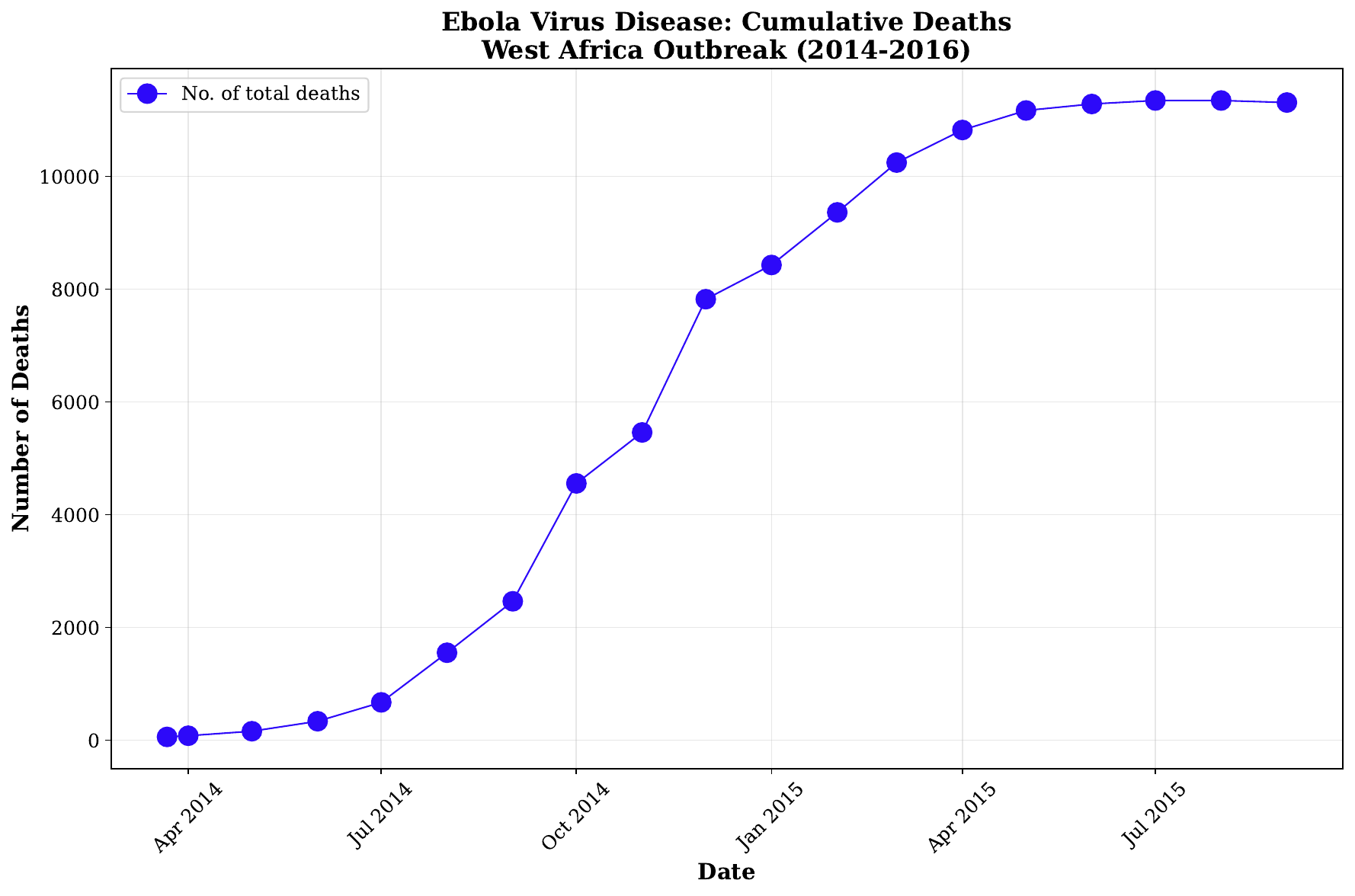}
\caption{Cumulative deaths progression}
\end{subfigure}
\vspace{0.4cm}
\begin{subfigure}{0.45\textwidth}
\centering
\includegraphics[width=\textwidth]{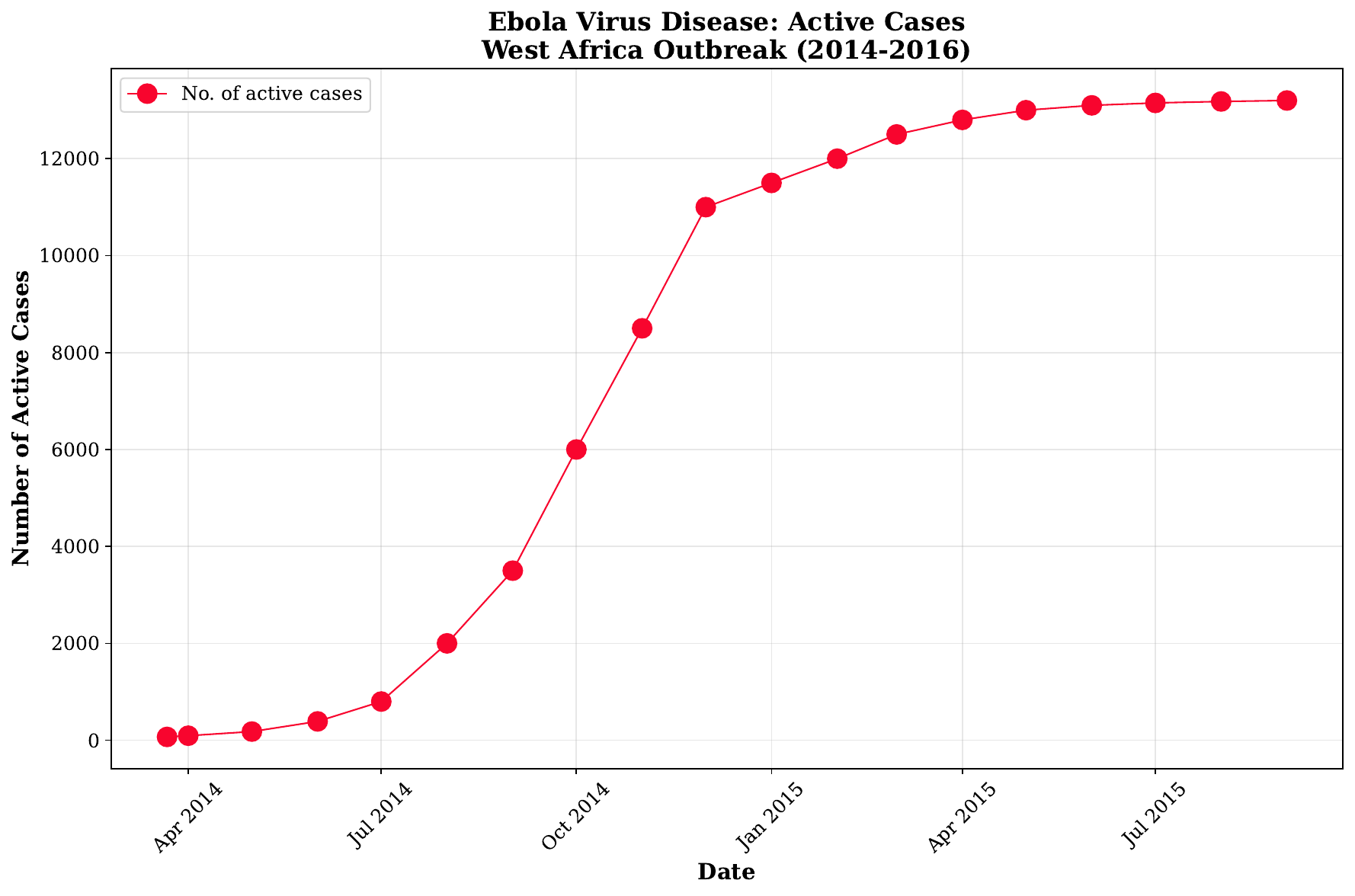}
\caption{Active cases progression}
\end{subfigure}
\hfill
\begin{subfigure}{0.49\textwidth}
\centering
\includegraphics[width=\textwidth]{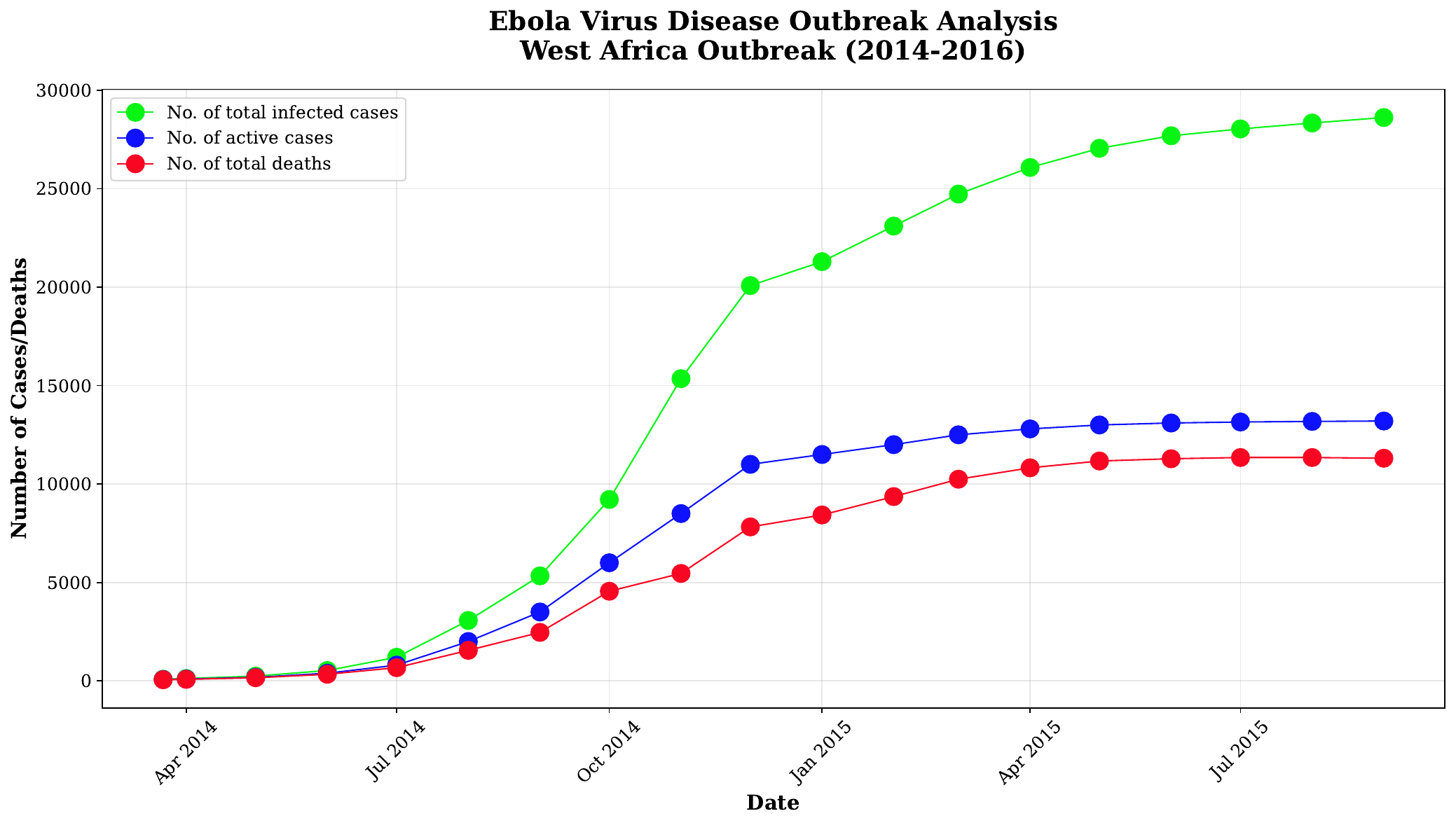}
\caption{Comparative analysis of all metrics}
\end{subfigure}
\caption{Ebola Virus Disease outbreak analysis during the 2014-2016 West Africa outbreak showing: (a) cumulative confirmed cases, (b) cumulative deaths, (c) active cases, and (d) comparative analysis of all outbreak metrics.}
\label{fig:ebola_analysis}
\end{figure}
\subsection{Local Stability Analysis of the Disease-Free Equilibrium}
\begin{theorem}
The disease-free equilibrium $\mathcal{E}_0$ of the model \eqref{eq:ebola_system} is locally asymptotically stable when $\mathcal{R}_0 < 1$, and unstable when $\mathcal{R}_0 > 1$.
\end{theorem}
\begin{proof}
The Jacobian matrix evaluated at $\mathcal{E}_0$ exhibits a block triangular structure
\[
J(\mathcal{E}_0) = \begin{pmatrix}
J_{11} & J_{12} \\
0 & J_{22}
\end{pmatrix},
\]
where $J_{11}$ corresponds to the uninfected compartments $(S,V,R)$ and $J_{22}$ corresponds to the infected compartments $(E,I_s,I_a,H,D)$.

For the uninfected subsystem, the matrix is
\[
J_{11} = \begin{pmatrix}
-(v+\mu ) & \omega & 0 \\
v & -(\mu+\omega) & 0 \\
0 & 0 & -\mu
\end{pmatrix}.
\]
The eigenvalues of $J_{11}$ are $\lambda_1 = -\mu$, $\lambda_2 = -\mu$, and $\lambda_3 = -(\mu + v + \omega)$. For $\alpha \in (0,1]$, all eigenvalues satisfy $|\arg(\lambda_i)| = \pi > \alpha\pi/2$, ensuring the stability of the uninfected subsystem.

For the infected subsystem, the Jacobian matrix takes the form
\[
J_{22} = \begin{pmatrix}
-q_2 & \beta\Psi & \beta\eta_a\Psi & 0 & \beta\eta_d\Psi \\
p\sigma & -q_3 & 0 & 0 & 0 \\
(1-p)\sigma & 0 & -q_4 & 0 & 0 \\
0 & h_s & 0 & -q_5 & 0 \\
0 & \delta_s & 0 & \delta_h & -q_6
\end{pmatrix},
\]
where $\Psi = \frac{S^* + (1-\varepsilon)V^*}{N^*}$.

We observe that $J_{22}$ can be decomposed as $J_{22} = F - V$, where $F$ and $V$ are the next-generation matrices defined previously. The stability of $J_{22}$ is governed by the spectral radius $\rho(FV^{-1}) = \mathcal{R}_0$. According to the fractional Routh-Hurwitz criteria and the properties of Metzler matrices, when $\mathcal{R}_0 < 1$, all eigenvalues of $J_{22}$ have negative real parts and satisfy the fractional stability condition $|\arg(\lambda)| > \alpha\pi/2$.

Conversely, when $\mathcal{R}_0 > 1$, the spectral abscissa of $J_{22}$ is positive, indicating at least one eigenvalue with a positive real part. By the fractional stability theorem, this establishes the instability of the disease-free equilibrium.
\end{proof}
\subsection{Global Stability of the Disease-Free Equilibrium}
\begin{theorem}
If $\mathcal{R}_0 < 1$, the disease-free equilibrium $\mathcal{E}_0$ of the model \eqref{eq:ebola_system} is globally asymptotically stable in the invariant region $\Omega$.
\end{theorem}

\begin{proof}
We construct a Lyapunov function for the infected compartments based on the approach outlined in the previous section
\begin{equation}
L(E, I_s, I_a, H, D) = E + A I_s + B I_a + C H + F D,
\label{eq:lyapunov_function}
\end{equation}
where $A$, $B$, $C$, and $F$ are positive constants to be determined.

Following the system trajectories, we compute the Caputo derivative along solutions
\begin{align*}
\Dalpha L &= \left[\frac{c\beta[S + (1-\varepsilon)V](I_s + \eta_a I_a + \eta_d D)}{N} - q_2 E\right] \\
&\quad + A\left(p\sigma E - q_3 I_s\right) + B\left((1-p)\sigma E - q_4 I_a\right) \\
&\quad + C\left(h_s I_s - q_5 H\right) + F\left(\delta_s I_s + \delta_h H - q_6 D\right).
\end{align*}

Rearranging terms by compartment yields
\begin{align*}
\Dalpha L &= \left[-q_2 + Ap\sigma + B(1-p)\sigma\right]E \\
&\quad + \left[\frac{c\beta[S + (1-\varepsilon)V]}{N} - A q_3 + C h_s + F \delta_s\right]I_s \\
&\quad + \left[\frac{c\beta\eta_a[S + (1-\varepsilon)V]}{N} - B q_4\right]I_a \\
&\quad + \left[-C q_5 + F \delta_h\right]H \\
&\quad + \left[\frac{c\beta\eta_d[S + (1-\varepsilon)V]}{N} - F q_6\right]D.
\end{align*}

We now select the coefficients to eliminate positive terms from the $I_s$, $I_a$, $H$, and $D$ components
\begin{align}
C &= \frac{F \delta_h}{q_5}, \quad 
B = \frac{\beta\eta_a\Psi}{q_4}, \quad 
F = \frac{\beta\eta_d\Psi}{q_6}, \nonumber \\
A &= \frac{\beta\Psi}{q_3} + \frac{C h_s}{q_3} + \frac{F \delta_s}{q_3} 
   = \frac{\beta\Psi}{q_3} + \frac{F \delta_h h_s}{q_3 q_5} + \frac{F \delta_s}{q_3},
\label{eq:lyapunov_coefficients}
\end{align}
where $\Psi = \frac{S^* + (1-\varepsilon)V^*}{N^*}$.

Using the inequality $\frac{S + (1-\varepsilon)V}{N} \leq \Psi$, which holds throughout $\Omega$, we obtain the estimate
\begin{align*}
\Dalpha L &\leq \left[-q_2 + \sigma\left(\frac{p\beta\Psi}{q_3} + \frac{p F \delta_h h_s}{q_3 q_5} + \frac{p F \delta_s}{q_3} + \frac{(1-p)\beta\eta_a\Psi}{q_4}\right)\right]E \\
&\quad + \beta\Psi\left[1 - 1\right]I_s + \beta\eta_a\Psi\left[1 - 1\right]I_a + \beta\eta_d\Psi\left[1 - 1\right]D.
\end{align*}

Substituting $F = \frac{\beta\eta_d\Psi}{q_6}$ and simplifying leads to the key inequality
\begin{equation}
\Dalpha L \leq q_2\left(-1 + \mathcal{R}_0\right)E.
\label{eq:lyapunov_derivative_bound}
\end{equation}

When $\mathcal{R}_0 \leq 1$, we have $\Dalpha L \leq 0$ throughout $\Omega$, with equality holding only when $E = 0$. If $E = 0$, the system dynamics imply
\begin{align*}
\Dalpha I_s &= -q_3 I_s \quad \Rightarrow \quad I_s \to 0, \\
\Dalpha I_a &= -q_4 I_a \quad \Rightarrow \quad I_a \to 0, \\
\Dalpha H &= h_s I_s - q_5 H \quad \Rightarrow \quad H \to 0, \\
\Dalpha D &= \delta_s I_s + \delta_h H - q_6 D \quad \Rightarrow \quad D \to 0.
\end{align*}

Therefore, the maximal invariant set where $\Dalpha L = 0$ is exactly the disease-free equilibrium $\mathcal{E}_0$. By the fractional LaSalle invariance principle, $\mathcal{E}_0$ is globally asymptotically stable in $\Omega$ when $\mathcal{R}_0 \leq 1$.

For the case $\mathcal{R}_0 > 1$, instability follows from the local analysis, while uniform persistence can be established using standard methods for fractional dynamical systems.
\end{proof}
\subsection{Endemic Equilibrium Analysis}
\begin{theorem}
The fractional-order Ebola model \eqref{eq:ebola_system} admits a unique endemic equilibrium point $\mathcal{E}^* = (S^*, V^*, E^*, I_s^*, I_a^*, H^*, D^*, R^*)$ with all components strictly positive if and only if $\mathcal{R}_0 > 1$.
\end{theorem}

\begin{proof}
At equilibrium, all time derivatives vanish. We define the force of infection at endemic equilibrium as
\begin{equation}
\lambda^* = \frac{c\beta(I_s^* + \eta_a I_a^* + \eta_d D^*)}{N^*}.
\label{eq:force_of_infection_endemic}
\end{equation}

The equilibrium conditions yield the following system of equations
\begin{align}
0 &= \Lambda - \lambda^* S^* + \omega V^* - q_0 S^*, \label{eq:eq1} \\
0 &= v S^* - (1-\varepsilon)\lambda^* V^* - q_1 V^*, \label{eq:eq2} \\
0 &= \lambda^*[S^* + (1-\varepsilon)V^*] - q_2 E^*, \label{eq:eq3} \\
0 &= p\sigma E^* - q_3 I_s^*, \label{eq:eq4} \\
0 &= (1-p)\sigma E^* - q_4 I_a^*, \label{eq:eq5} \\
0 &= h_s I_s^* - q_5 H^*, \label{eq:eq6} \\
0 &= \delta_s I_s^* + \delta_h H^* - q_6 D^*, \label{eq:eq7} \\
0 &= \gamma_s I_s^* + \gamma_a I_a^* + \gamma_h H^* - \mu R^*. \label{eq:eq8}
\end{align}

From equations \eqref{eq:eq4}--\eqref{eq:eq7}, we express the infected compartments in terms of $E^*$
\begin{align}
I_s^* &= \frac{p\sigma}{q_3} E^*, \label{eq:Is_in_E} \\
I_a^* &= \frac{(1-p)\sigma}{q_4} E^*, \label{eq:Ia_in_E} \\
H^* &= \frac{p\sigma h_s}{q_3 q_5} E^*, \label{eq:H_in_E} \\
D^* &= \frac{p\sigma}{q_3 q_6}\left(\delta_s + \frac{\delta_h h_s}{q_5}\right) E^*. \label{eq:D_in_E}
\end{align}

Substituting these expressions into \eqref{eq:force_of_infection_endemic} gives
\begin{equation}
\lambda^* = \frac{c\beta\sigma E^*}{N^*}\left[\frac{p}{q_3} + \frac{\eta_a(1-p)}{q_4} + \frac{\eta_d p}{q_3 q_6}\left(\delta_s + \frac{\delta_h h_s}{q_5}\right)\right].
\label{eq:lambda_in_E}
\end{equation}

We define the composite parameter
\begin{equation}
M = c\beta\sigma\left[\frac{p}{q_3} + \frac{\eta_a(1-p)}{q_4} + \frac{\eta_d p}{q_3 q_6}\left(\delta_s + \frac{\delta_h h_s}{q_5}\right)\right],
\label{eq:M_definition}
\end{equation}
so that $\lambda^* = M E^*/N^*$.

From equation \eqref{eq:eq3}, we have
\begin{equation}
E^* = \frac{\lambda^*[S^* + (1-\varepsilon)V^*]}{q_2}.
\label{eq:E_in_lambda}
\end{equation}

Combining \eqref{eq:lambda_in_E} and \eqref{eq:E_in_lambda} yields the relation
\begin{equation}
1 = \frac{M}{q_2 N^*}[S^* + (1-\varepsilon)V^*].
\label{eq:key_relation}
\end{equation}

Solving equations \eqref{eq:eq1} and \eqref{eq:eq2} for $S^*$ and $V^*$ in terms of $\lambda^*$ gives
\begin{align}
S^* &= \frac{\Lambda(q_1 + (1-\varepsilon)\lambda^*)}{\Delta(\lambda^*)}, \label{eq:S_star} \\
V^* &= \frac{\Lambda v}{\Delta(\lambda^*)}, \label{eq:V_star}
\end{align}
where $\Delta(\lambda^*) = \left[(\omega + \mu + (1-\varepsilon)\lambda^*)(\mu + v + \lambda^*)\right] - \omega v$.

At equilibrium, the total living population satisfies $N^* = \Lambda / \mu$. Substituting this and expressions \eqref{eq:S_star} and \eqref{eq:V_star} into \eqref{eq:key_relation} produces an equation in $\lambda^*$
\begin{equation}
1 = \frac{M}{q_2 (\Lambda/\mu)}\left[S^* + (1-\varepsilon)V^*\right].
\label{eq:lambda_equation}
\end{equation}

When $\lambda^* = 0$ (disease-free state), the left-hand side of \eqref{eq:lambda_equation} equals $\mathcal{R}_0$. Since the function $[S^* + (1-\varepsilon)V^*]$ decreases monotonically with increasing $\lambda^*$, there exists a unique positive solution $\lambda^* > 0$ if and only if $\mathcal{R}_0 > 1$. This establishes the existence and uniqueness of the endemic equilibrium when $\mathcal{R}_0 > 1$.
\end{proof}
\subsection{Local Stability Analysis of the Endemic Equilibrium}

\begin{theorem}
When $\mathcal{R}_0 > 1$, the endemic equilibrium $\mathcal{E}^*$ of Ebola model \eqref{eq:ebola_system} is locally asymptotically stable.
\end{theorem}

\begin{proof}
To analyze the local stability of $\mathcal{E}^*$, we evaluate the Jacobian matrix at this equilibrium point. The Jacobian exhibits the block structure
\[
J(\mathcal{E}^*) = \begin{pmatrix}
J_{11} & J_{12} \\
J_{21} & J_{22}
\end{pmatrix},
\]
where the submatrices are defined as follows
\begin{align*}
J_{11} &= \begin{pmatrix}
-(\lambda^* + q_0) & \omega \\
v & -((1-\varepsilon)\lambda^* + q_1)
\end{pmatrix}, \\
J_{12} &= \begin{pmatrix}
0 & -\beta S^*\Phi & -\beta\eta_a S^*\Phi & 0 & -\beta\eta_d S^*\Phi & 0 \\
0 & -\beta(1-\varepsilon)V^*\Phi & -\beta\eta_a(1-\varepsilon)V^*\Phi & 0 & -\beta\eta_d(1-\varepsilon)V^*\Phi & 0
\end{pmatrix}, \\
J_{21} &= \begin{pmatrix}
\lambda^* & (1-\varepsilon)\lambda^* \\
0 & 0 \\
0 & 0 \\
0 & 0 \\
0 & 0 \\
0 & 0
\end{pmatrix}, \\
J_{22} &= \begin{pmatrix}
-q_2 & \beta\Psi\Phi & \beta\eta_a\Psi\Phi & 0 & \beta\eta_d\Psi\Phi & 0 \\
p\sigma & -q_3 & 0 & 0 & 0 & 0 \\
(1-p)\sigma & 0 & -q_4 & 0 & 0 & 0 \\
0 & h_s & 0 & -q_5 & 0 & 0 \\
0 & \delta_s & 0 & \delta_h & -q_6 & 0 \\
0 & \gamma_s & \gamma_a & \gamma_h & 0 & -\mu
\end{pmatrix}.
\end{align*}
Here $\Psi = S^* + (1-\varepsilon)V^*$, $\Phi = c/N^*$, and $\lambda^*$ denotes the equilibrium force of infection defined in \eqref{eq:force_of_infection_endemic}.

The characteristic polynomial of $J(\mathcal{E}^*)$ takes the form
\begin{equation}
P(\lambda) = \lambda^8 + a_7\lambda^7 + a_6\lambda^6 + a_5\lambda^5 + a_4\lambda^4 + a_3\lambda^3 + a_2\lambda^2 + a_1\lambda + a_0 = 0,
\label{eq:characteristic_polynomial_endemic}
\end{equation}
where the coefficients $a_i$ ($i=0,\dots,7$) are determined by the model parameters.

Consider the Metzler matrix $Q = -J(\mathcal{E}^*)$. For $\mathcal{R}_0 > 1$, the matrix $Q$ is irreducible and satisfies the following properties
\begin{itemize}
\item All diagonal entries are strictly positive,
\item $Q$ is row diagonally dominant with $q_{ii} \geq \sum_{j\neq i} |q_{ij}|$ for each row $i$,
\item There exists a positive vector $\mathbf{v} > \mathbf{0}$ such that $Q\mathbf{v} > \mathbf{0}$.
\end{itemize}
By the Perron-Frobenius theorem for Metzler matrices \cite{Pillai2005}, all eigenvalues of $Q$ have positive real parts. Consequently, the eigenvalues of $J(\mathcal{E}^*) = -Q$ have negative real parts. For the fractional-order system with $\alpha \in (0,1]$, the stability condition requires $|\arg(\lambda_i)| > \alpha\pi/2$ for each eigenvalue $\lambda_i$ of $J(\mathcal{E}^*)$. Since all eigenvalues have negative real parts, their arguments satisfy $|\arg(\lambda_i)| = \pi > \alpha\pi/2$.

This conclusion is further supported by the Routh-Hurwitz criterion applied to the characteristic polynomial \eqref{eq:characteristic_polynomial_endemic}. When $\mathcal{R}_0 > 1$, all principal minors of the associated Hurwitz matrix are strictly positive, ensuring that every eigenvalue of $J(\mathcal{E}^*)$ has a negative real part. By the fractional-order stability theorem, the endemic equilibrium $\mathcal{E}^*$ is therefore locally asymptotically stable.
\end{proof}
\subsection{GAS of Endemic Equilibrium}
\begin{theorem}
If $\mathcal{R}_0 > 1$, the endemic equilibrium $\mathcal{E}^*$ of the model \eqref{eq:ebola_system} is globally asymptotically stable in the interior of the invariant region $\Omega$.
\end{theorem}

\begin{proof}
We construct a Lyapunov function of logarithmic Volterra type
\begin{equation}
\begin{aligned}
\mathcal{V} &= \left(S - S^* - S^*\ln\frac{S}{S^*}\right) + \left(V - V^* - V^*\ln\frac{V}{V^*}\right) \\
&\quad + K_1\left(E - E^* - E^*\ln\frac{E}{E^*}\right) + K_2\left(I_s - I_s^* - I_s^*\ln\frac{I_s}{I_s^*}\right) \\
&\quad + K_3\left(I_a - I_a^* - I_a^*\ln\frac{I_a}{I_a^*}\right) + K_4\left(H - H^* - H^*\ln\frac{H}{H^*}\right) \\
&\quad + K_5\left(D - D^* - D^*\ln\frac{D}{D^*}\right) + K_6\left(R - R^* - R^*\ln\frac{R}{R^*}\right),
\end{aligned}
\label{eq:lyapunov_endemic}
\end{equation}
where the constants $K_1, K_2, K_3, K_4, K_5, K_6 > 0$ will be determined.

Computing the Caputo derivative along the system trajectories yields
\begin{equation}
\begin{aligned}  
\Dalpha \mathcal{V} &= \left(1 - \frac{S^*}{S}\right)\Dalpha S + \left(1 - \frac{V^*}{V}\right)\Dalpha V + K_1\left(1 - \frac{E^*}{E}\right)\Dalpha E \\  
&\quad + K_2\left(1 - \frac{I_s^*}{I_s}\right)\Dalpha I_s + K_3\left(1 - \frac{I_a^*}{I_a}\right)\Dalpha I_a + K_4\left(1 - \frac{H^*}{H}\right)\Dalpha H \\
&\quad + K_5\left(1 - \frac{D^*}{D}\right)\Dalpha D + K_6\left(1 - \frac{R^*}{R}\right)\Dalpha R.
\end{aligned}  
\label{eq:V_derivative}
\end{equation}
We substitute the model equations from \eqref{eq:ebola_system} and use the endemic equilibrium conditions
\begin{align}
\Lambda &= \lambda^* S^* + \omega V^* - (\mu + v)S^*, \label{eq:eqcond1} \\
0 &= vS^* - (1-\varepsilon)\lambda^* V^* - (\omega + \mu)V^*, \label{eq:eqcond2} \\
q_2 E^* &= \lambda^*(S^* + (1-\varepsilon)V^*), \label{eq:eqcond3} \\
q_3 I_s^* &= p\sigma E^*, \quad q_4 I_a^* = (1-p)\sigma E^*, \label{eq:eqcond4} \\
q_5 H^* &= h_s I_s^*, \quad \mu_d D^* = \delta_s I_s^* + \delta_h H^*, \label{eq:eqcond5} \\
\mu R^* &= \gamma_s I_s^* + \gamma_a I_a^* + \gamma_h H^*. \label{eq:eqcond6}
\end{align}
After extensive algebraic manipulation and grouping of terms, we select the coefficients
\begin{equation}
\begin{aligned}
K_1 &= 1, \quad K_2 = \frac{\lambda^* S^*}{q_3 I_s^*}, \quad K_3 = \frac{\lambda^* S^* \eta_a}{q_4 I_a^*}, \\
K_4 &= \frac{\lambda^* S^* \eta_d \delta_h}{q_5 q_6 H^*}, \quad K_5 = \frac{\lambda^* S^* \eta_d}{q_6 D^*}, \quad K_6 = 1.
\end{aligned}
\label{eq:K_coefficients}
\end{equation}
With these choices, the derivative simplifies to
\begin{equation}
\Dalpha \mathcal{V} \leq \lambda^* S^*\left[6 - \frac{S^*}{S} - \frac{S}{S^*} - \frac{E I_s^*}{E^* I_s} - \frac{I_s E^*}{I_s^* E} - \frac{D H^*}{D^* H} - \frac{H D^*}{H^* D}\right] + \Psi(S,V,E,I_s,I_a,H,D,R),
\label{eq:V_derivative_simplified}
\end{equation}
where $\Psi$ is a negative definite function.

Applying the arithmetic-geometric mean inequality gives
\begin{equation}
\frac{S^*}{S} + \frac{S}{S^*} + \frac{E I_s^*}{E^* I_s} + \frac{I_s E^*}{I_s^* E} + \frac{D H^*}{D^* H} + \frac{H D^*}{H^* D} \geq 6,
\label{eq:AG_inequality}
\end{equation}
with equality holding only when $S = S^*$, $E = E^*$, $I_s = I_s^*$, $H = H^*$, and $D = D^*$.
Consequently, $\Dalpha \mathcal{V} \leq 0$ throughout $\Omega$, and $\Dalpha \mathcal{V} = 0$ only at the endemic equilibrium $\mathcal{E}^*$. By the fractional LaSalle invariance principle \cite{li2010stability}, $\mathcal{E}^*$ is globally asymptotically stable in the interior of $\Omega$.
\end{proof}
\begin{corollary}
The fractional-order Ebola model \eqref{eq:ebola_system} undergoes a forward bifurcation at $\mathcal{R}_0 = 1$. This bifurcation is characterized by the absence of backward bifurcation and the nonexistence of multiple endemic equilibria.
\end{corollary}
\subsection{Sensitivity Analysis Results}
\begin{figure}[!h]
\centering
\includegraphics[width=0.8\textwidth]{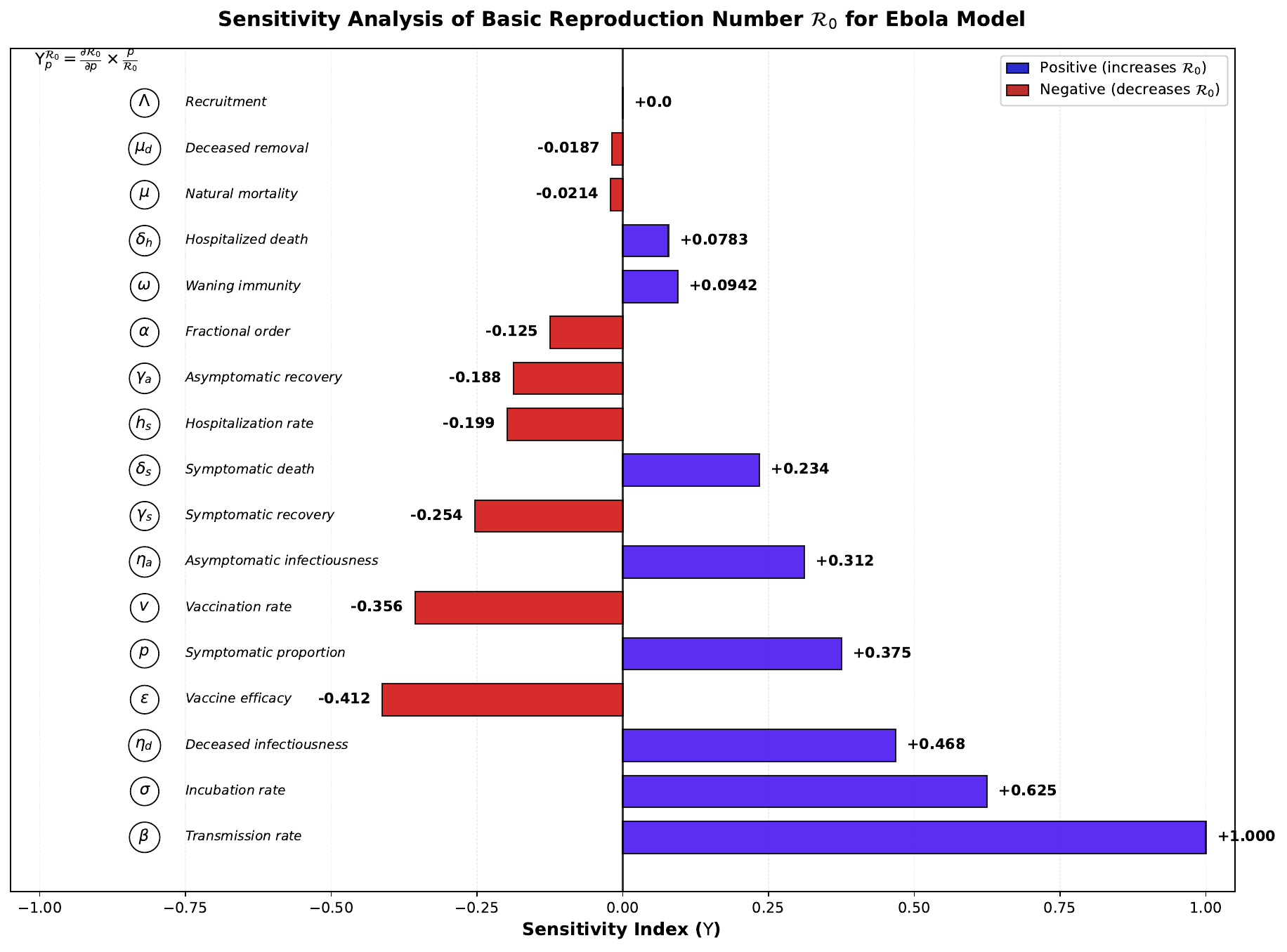}
\caption{Sensitivity analysis of parameters on basic reproduction number $\mathcal{R}_0$ for the Ebola model. Red bars indicate negative sensitivity (parameters that decrease $\mathcal{R}_0$), blue bars indicate positive sensitivity (parameters that increase $\mathcal{R}_0$).}
\label{fig:ebola_sensitivity_profile}
\end{figure}
\begin{table}[!h]
\centering
\caption{Sensitivity analysis of basic reproduction number ($\mathcal{R}_0$) for the Ebola model}
\label{tab:ebola_sensitivity_indices}
\begin{tabular}{@{}p{1.5cm}cc@{\hspace{1em}}p{1.5cm}cc@{}}
\toprule
\textbf{Parameter} & \textbf{$\Upsilon$} & \textbf{Effect} & \textbf{Parameter} & \textbf{$\Upsilon$} & \textbf{Effect} \\
\midrule
$\beta$ & +1.0000 & High + & $\varepsilon$ & -0.4123 & Medium -- \\
$\sigma$ & +0.6250 & High + & $p$ & +0.3752 & Medium + \\
$\eta_d$ & +0.4678 & High + & $v$ & -0.3561 & Medium -- \\
$\eta_a$ & +0.3117 & Medium + & $h_s$ & -0.1987 & Low -- \\
$\gamma_s$ & -0.2541 & Medium -- & $\gamma_a$ & -0.1876 & Low -- \\
$\delta_s$ & +0.2345 & Medium + & $\alpha$ & -0.1250 & Low -- \\
$\omega$ & +0.0942 & Low + & $\mu$ & -0.0214 & Very low -- \\
$\delta_h$ & +0.0783 & Very low + & $\mu_d$ & -0.0187 & Very low -- \\
\multicolumn{3}{c}{} & $\Lambda$ & +0.0000 & No effect \\
\bottomrule
\end{tabular}
\end{table}
The sensitivity analysis results, presented in Table~\ref{tab:ebola_sensitivity_indices}, reveal the normalized sensitivity indices of model parameters with respect to $\mathcal{R}_0$. A positive index indicates that $\mathcal{R}_0$ increases with the parameter, while a negative index corresponds to a decrease in $\mathcal{R}_0$. Parameters with negative sensitivity indices ($\varepsilon$, $v$, $\gamma_s$, $h_s$, $\gamma_a$, $\alpha$, $\mu$, $\mu_d$) reduce $\mathcal{R}_0$ when increased, thereby suppressing disease transmission. Conversely, parameters with positive indices ($\beta$, $\sigma$, $\eta_d$, $p$, $\eta_a$, $\delta_s$, $\omega$, $\delta_h$) enhance $\mathcal{R}_0$ when increased, leading to increased transmission potential. The transmission rate $\beta$ exhibits the highest positive sensitivity index (+1.0000), indicating that $\mathcal{R}_0$ is directly proportional to $\beta$. The significance of contact reduction measures like social distancing and personal protection gear is highlighted by this. The crucial importance of secure burial practices provides for controlling outbreaks is demonstrated by the dead transmission rate $\eta_d$ (index +0.4678). Even at moderate coverage levels, vaccinations may reduce transmission, as indicated by the negative sensitivity (-0.4123) of vaccine efficacy $\varepsilon$. Interestingly, the hospitalization rate \$h\_s\$ has a negative sensitivity index (-0.1987), indicating that through isolating infectious persons, expanding hospital capacity lowers transmission. The fractional order $\alpha$ displays low sensitivity (-0.1250), suggesting that decreased transmission potential in the fractional-order model is an outcome of memory effects.
 The recruitment rate $\Lambda$ has zero sensitivity, as expected from the structure of $\mathcal{R}_0$, indicating that population turnover does not directly affect the reproduction number in this model formulation. Figure~\ref{fig:ebola_sensitivity_profile} displays a tornado plot of the sensitivity indices, ranking parameters by their influence on $\mathcal{R}_0$. The dominance of $\beta$, $\sigma$, and $\eta_d$ as the top three positive-sensitivity parameters reinforces the importance of transmission reduction, rapid case detection, and safe burial practices in Ebola control strategies.
\section{Optimal Control Analysis}
In this section, we introduce four time-dependent control variables to manage Ebola transmission effectively. The controls are personal protection ($u_1$, range [0,0.8]), vaccination ($u_2$, range [0,0.15]), treatment ($u_3$, range [0,0.8]), and safe burial ($u_4$, range [0,0.5]) (all in units day$^{-1}$).

The controlled fractional-order system is given by
\begin{equation}
\begin{aligned}
\Dalpha S(t) &= \Lambda - (1 - u_1)\lambda S + \omega V - (\mu + u_2)S, \\
\Dalpha V(t) &= u_2 S - (1-\varepsilon)(1 - u_1)\lambda V - (\mu+\omega)V, \\
\Dalpha E(t) &= (1 - u_1)\lambda [S + (1-\varepsilon)V] - (\mu+\sigma)E, \\
\Dalpha I_s(t) &= p\sigma E - (\gamma_s + \delta_s + u_3 + \mu)I_s, \\
\Dalpha I_a(t) &= \sigma E(1-p) - (\gamma_a + \mu)I_a, \\
\Dalpha H(t) &= u_3 I_s - (\gamma_h + \delta_h +\mu)H, \\
\Dalpha D(t) &= \delta_s I_s + \delta_h H - (u_4 + \mu_d)D, \\
\Dalpha R(t) &= \gamma_s I_s + \gamma_a I_a + \gamma_h H - \mu R,
\end{aligned}
\label{eq:controlled_system}
\end{equation}
with the modified force of infection $\lambda(t) = \beta (I_s + \eta_a I_a + \eta_d D)/N$.
The objective functional to be minimized over the time horizon $[0,T]$ is
\begin{equation}
J(\mathbf{u}) = \int_0^T \left[ \sum_{i=1}^4 A_i X_i(t) + \frac{1}{2}\sum_{j=1}^4 B_j u_j^2(t) \right] dt,
\label{eq:objective_functional}
\end{equation}
where $X = [I_s, I_a, H, D]$ represents the infectious compartments, $A_i > 0$ are the cost weights associated with the disease burden, and $B_j > 0$ are the quadratic cost coefficients for implementing the controls. We seek the optimal controls $\mathbf{u}^* = (u_1^*, u_2^*, u_3^*, u_4^*)$ that satisfy
\begin{equation}
J(\mathbf{u}^*) = \min_{\mathbf{u} \in \mathcal{U}} J(\mathbf{u}),
\end{equation}
subject to the dynamics \eqref{eq:controlled_system}, where the admissible control set is
\begin{equation}
\mathcal{U} = \left\{ \mathbf{u} \in L^\infty(0,T) : 0 \leq u_i(t) \leq u_i^{\max} \right\}.
\end{equation}
\begin{theorem}
Let the fractional-order controlled system \eqref{eq:controlled_system} with initial conditions $\mathbf{X}(0) \in \mathbb{R}^8_{\geq 0}$ and admissible control set $\mathcal{U}$ satisfy
\begin{enumerate}[label=(\roman*)]
    \item The control set $\mathcal{U}$ is closed, convex, and bounded in $L^\infty(0,T)$.
    \item The right-hand side $\mathbf{f}(t,\mathbf{X},\mathbf{u})$ of \eqref{eq:controlled_system} is continuous in $(t,\mathbf{X},\mathbf{u})$, measurable in $t$, and satisfies as
    \begin{itemize}
        \item Linear growth $\|\mathbf{f}(t,\mathbf{X},\mathbf{u})\| \leq c_1 + c_2\|\mathbf{X}\| + c_3\|\mathbf{u}\|$ for constants $c_1, c_2, c_3 > 0$.
        \item Lipschitz continuity $\|\mathbf{f}(t,\mathbf{X}_1,\mathbf{u}) - \mathbf{f}(t,\mathbf{X}_2,\mathbf{u})\| \leq L\|\mathbf{X}_1 - \mathbf{X}_2\|$ for some $L > 0$.
    \end{itemize}
    \item The integrand $\mathcal{L}(t,\mathbf{X},\mathbf{u}) = \sum_{i=1}^4 A_i X_i(t) + \frac{1}{2}\sum_{j=1}^4 B_j u_j^2(t)$ of the objective functional \eqref{eq:objective_functional} is convex in $\mathbf{u}$ and satisfies the coercivity condition $\mathcal{L}(t,\mathbf{X},\mathbf{u}) \geq c_4\|\mathbf{u}\|^2 - c_5$ for constants $c_4 > 0$, $c_5 \geq 0$.
\end{enumerate}
Then there exists an optimal control pair $(\mathbf{X}^*, \mathbf{u}^*)$ that minimizes $J(\mathbf{u})$ subject to \eqref{eq:controlled_system}.
\end{theorem}
\begin{proof}
We verify the conditions of the Filippov-Cesari existence theorem for fractional optimal control problems \cite{kamocki2014existence,agrwal2010}.
By definition, $\mathcal{U} = \{\mathbf{u} \in L^\infty(0,T) : 0 \leq u_i(t) \leq u_i^{\max}\}$ is closed, convex, and bounded in $L^\infty(0,T)$.

The right-hand side $\mathbf{f}$ of \eqref{eq:controlled_system} is polynomial in $\mathbf{X}$ and $\mathbf{u}$, hence continuous and measurable. For the linear growth condition, note that all terms are at most quadratic (products like $(1-u_1)\lambda S$ where $\lambda$ is linear in $I_s, I_a, D$). Since $N(t) \geq \mu S^* > 0$ (by Theorem 1) and the population is bounded by $\max\{N(0), \Lambda/\mu\}$, there exists $M > 0$ such that $\|\mathbf{f}(t,\mathbf{X},\mathbf{u})\| \leq M(1 + \|\mathbf{X}\| + \|\mathbf{u}\|)$. Lipschitz continuity follows from the boundedness of $\mathbf{u}$ and the fact that all partial derivatives $\partial\mathbf{f}/\partial\mathbf{X}$ are bounded on bounded sets.

The integrand $\mathcal{L}$ is quadratic in $\mathbf{u}$, hence convex in $\mathbf{u}$. Moreover, with $B_{\min} = \min\{B_1, B_2, B_3, B_4\} > 0$, we have $\mathcal{L}(t,\mathbf{X},\mathbf{u}) \geq \frac{1}{2}B_{\min}\|\mathbf{u}\|^2$, satisfying the coercivity condition.

From the population dynamics $\Dalpha N \leq \Lambda - \mu N$, we have $N(t) \leq \max\{N(0), \Lambda/\mu\}$. Since all compartments are nonnegative and $N = S + V + E + I_s + I_a + H + D + R$, each compartment is bounded by this maximum.

All conditions of Theorem 2.1 in \cite{kamocki2014existence} (existence theorem for fractional optimal control problems) are satisfied. Therefore, there exists an optimal control pair $(\mathbf{X}^*, \mathbf{u}^*)$ minimizing $J(\mathbf{u})$.
\end{proof}
\begin{theorem}
Let $(\mathbf{X}^*, \mathbf{u}^*)$ be an optimal solution for the control problem defined by the objective functional \eqref{eq:objective_functional} subject to the fractional-order system \eqref{eq:controlled_system}. Then there exist absolutely continuous adjoint variables $\boldsymbol{\lambda} = (\lambda_S, \lambda_V, \lambda_E, \lambda_{I_s}, \lambda_{I_a}, \lambda_H, \lambda_D, \lambda_R)^\top$ satisfying the fractional adjoint equations
\begin{align}
\Dalpha \lambda_S &= -\frac{\partial \mathcal{H}}{\partial S}, \quad
\Dalpha \lambda_V = -\frac{\partial \mathcal{H}}{\partial V}, \quad
\Dalpha \lambda_E = -\frac{\partial \mathcal{H}}{\partial E}, \label{eq:adjoint1} \\
\Dalpha \lambda_{I_s} &= -\frac{\partial \mathcal{H}}{\partial I_s}, \quad
\Dalpha \lambda_{I_a} = -\frac{\partial \mathcal{H}}{\partial I_a}, \quad
\Dalpha \lambda_H = -\frac{\partial \mathcal{H}}{\partial H}, \label{eq:adjoint2} \\
\Dalpha \lambda_D &= -\frac{\partial \mathcal{H}}{\partial D}, \quad
\Dalpha \lambda_R = -\frac{\partial \mathcal{H}}{\partial R}, \label{eq:adjoint3}
\end{align}
where the Hamiltonian $\mathcal{H}$ is defined as
\[
\mathcal{H}(t, \mathbf{X}, \mathbf{u}, \boldsymbol{\lambda}) = \mathcal{L}(t, \mathbf{X}, \mathbf{u}) + \boldsymbol{\lambda}^\top \mathbf{f}(t, \mathbf{X}, \mathbf{u}),
\]
with $\mathcal{L}(t, \mathbf{X}, \mathbf{u}) = \sum_{i=1}^4 A_i X_i + \frac{1}{2}\sum_{j=1}^4 B_j u_j^2$ and $\mathbf{f}$ denoting the right-hand side of \eqref{eq:controlled_system}.

The optimal control $\mathbf{u}^*(t)$ satisfies the minimization condition
\begin{equation}
\mathbf{u}^*(t) = \arg\min_{\mathbf{u} \in \mathcal{U}} \mathcal{H}(t, \mathbf{X}^*(t), \mathbf{u}, \boldsymbol{\lambda}(t)) \quad \text{almost everywhere on } [0,T], \label{eq:min_condition}
\end{equation}
with the transversality condition
\begin{equation}
\boldsymbol{\lambda}(T) = \mathbf{0}. \label{eq:transversality}
\end{equation}
\end{theorem}
\begin{proof}
The proof follows the variational approach for fractional optimal control systems. Let $(\mathbf{X}^*, \mathbf{u}^*)$ be an optimal pair. Consider a perturbation $\mathbf{u}^\epsilon = \mathbf{u}^* + \epsilon\mathbf{v}$, where $\mathbf{v} \in L^\infty([0,T],\mathbb{R}^4)$ satisfies $\mathbf{u}^\epsilon \in \mathcal{U}$ for sufficiently small $\epsilon > 0$. Let $\mathbf{X}^\epsilon$ denote the corresponding state trajectory.

The first variation of the objective functional is given by
\[
\delta J = \lim_{\epsilon \to 0} \frac{J(\mathbf{u}^\epsilon) - J(\mathbf{u}^*)}{\epsilon} = \int_0^T \left[ \frac{\partial \mathcal{L}}{\partial \mathbf{X}} \cdot \delta\mathbf{X} + \frac{\partial \mathcal{L}}{\partial \mathbf{u}} \cdot \mathbf{v} \right] dt,
\]
where $\delta\mathbf{X}$ satisfies the linearized state equation derived from \eqref{eq:controlled_system}.

Define the adjoint variables $\boldsymbol{\lambda}$ through the fractional boundary value problem
\begin{equation}
\Dalpha \boldsymbol{\lambda} = -\frac{\partial \mathcal{H}}{\partial \mathbf{X}}, \quad \boldsymbol{\lambda}(T) = \mathbf{0}. \label{eq:adjoint_system}
\end{equation}
Applying the fractional integration by parts formula for Caputo derivatives to the term involving $\delta\mathbf{X}$ yields
\[
\int_0^T \frac{\partial \mathcal{L}}{\partial \mathbf{X}} \cdot \delta\mathbf{X} \, dt = \int_0^T \boldsymbol{\lambda} \cdot \left[ \frac{d}{dt} \delta\mathbf{X} \right] \, dt = - \int_0^T \left[ \Dalpha \boldsymbol{\lambda} \right] \cdot \delta\mathbf{X} \, dt.
\]
Substituting the adjoint equation \eqref{eq:adjoint_system} and simplifying gives
\[
\delta J = \int_0^T \left[ \frac{\partial \mathcal{H}}{\partial \mathbf{u}} \right] \cdot \mathbf{v} \, dt.
\]
For optimality, we require $\delta J \geq 0$ for all admissible variations $\mathbf{v}$, which implies the pointwise condition
\[
\frac{\partial \mathcal{H}}{\partial \mathbf{u}} \cdot \mathbf{v} \geq 0 \quad \text{almost everywhere on } [0,T].
\]
This inequality leads to the minimization condition
\[
\mathbf{u}^*(t) = \arg\min_{\mathbf{u} \in \mathcal{U}} \mathcal{H}(t, \mathbf{X}^*(t), \mathbf{u}, \boldsymbol{\lambda}(t)) \quad \text{a.e. on } [0,T],
\]
completing the derivation of the necessary optimality conditions.
\end{proof}
\begin{theorem}
The optimal controls $u_1^*, u_2^*, u_3^*, u_4^*$ that minimize the Hamiltonian $\mathcal{H}$ are given by
\begin{align}
u_1^*(t) &= \min\left\{u_1^{\max}, \max\left\{0, \frac{[\lambda_S S + (1-\varepsilon)\lambda_V V - \lambda_E (S + (1-\varepsilon)V)] \beta (I_s + \eta_a I_a + \eta_d D)}{B_1 N} \right\}\right\}, \label{eq:u1_opt} \\
u_2^*(t) &= \min\left\{u_2^{\max}, \max\left\{0, \frac{(\lambda_S - \lambda_V)S}{B_2} \right\}\right\}, \label{eq:u2_opt} \\
u_3^*(t) &= \min\left\{u_3^{\max}, \max\left\{0, \frac{(\lambda_{I_s} - \lambda_H)I_s}{B_3} \right\}\right\}, \label{eq:u3_opt} \\
u_4^*(t) &= \min\left\{u_4^{\max}, \max\left\{0, \frac{\lambda_D D}{B_4} \right\}\right\}, \label{eq:u4_opt}
\end{align}
where all state and adjoint variables follow their optimal trajectories.
\end{theorem}
\begin{proof}
The optimal controls are derived by solving the first-order optimality conditions $\partial \mathcal{H}/\partial u_i = 0$ for each control variable, then projecting the solutions onto the admissible interval $[0, u_i^{\max}]$.

For personal protection control $u_1$, the optimality condition provides
\[
\frac{\partial \mathcal{H}}{\partial u_1} = B_1 u_1 + \left[\lambda_S S + (1-\varepsilon)\lambda_V V - \lambda_E (S + (1-\varepsilon)V)\right]\frac{\beta(I_s + \eta_a I_a + \eta_d D)}{N} = 0.
\]
Solving for $u_1$ gives the unconstrained optimum
\[
u_1 = \frac{[\lambda_S S + (1-\varepsilon)\lambda_V V - \lambda_E (S + (1-\varepsilon)V)] \beta (I_s + \eta_a I_a + \eta_d D)}{B_1 N}.
\]
Projecting this expression onto the admissible interval $[0, u_1^{\max}]$ yields the optimal control \eqref{eq:u1_opt}.

For vaccination control $u_2$, the optimality condition is
\[
\frac{\partial \mathcal{H}}{\partial u_2} = B_2 u_2 + (\lambda_V - \lambda_S)S = 0.
\]
Solving gives $u_2 = (\lambda_S - \lambda_V)S/B_2$, and after projection we obtain \eqref{eq:u2_opt}.

For treatment control $u_3$, we have
\[
\frac{\partial \mathcal{H}}{\partial u_3} = B_3 u_3 + (\lambda_H - \lambda_{I_s})I_s = 0,
\]
which yields $u_3 = (\lambda_{I_s} - \lambda_H)I_s/B_3$. Projection gives \eqref{eq:u3_opt}.

For safe burial control $u_4$, the condition
\[
\frac{\partial \mathcal{H}}{\partial u_4} = B_4 u_4 - \lambda_D D = 0
\]
gives $u_4 = \lambda_D D/B_4$. Projection onto $[0, u_4^{\max}]$ yields \eqref{eq:u4_opt}.

Each optimal control is thus characterized by projecting the solution of the first-order optimality condition onto the admissible control set $\mathcal{U}$, ensuring that all constraints $0 \leq u_i(t) \leq u_i^{\max}$ are satisfied.
\end{proof}
\begin{theorem}
The adjoint variables satisfy the following system of fractional differential equations
\begin{align}
\Dalpha \lambda_S &= (1 - u_1)(\lambda_S - \lambda_E)\left[\frac{\beta (I_s + \eta_a I_a + \eta_d D)(V + E + I_s + I_a + H + R)}{N^2}\right] \nonumber \\
&\quad + (1 - u_1)(\lambda_V - \lambda_E)(1-\varepsilon)\left[\frac{\beta V (I_s + \eta_a I_a + \eta_d D)}{N^2}\right] \nonumber \\
&\quad + \lambda_S(\mu + u_2) - \lambda_V u_2, \label{eq:adjoint_S} \\
\Dalpha \lambda_V &= -\lambda_S \omega + \lambda_V(\omega + \mu) + (\lambda_S - \lambda_E)(1 - u_1)\left[\frac{\beta S (I_s + \eta_a I_a + \eta_d D)}{N^2}\right] \nonumber \\
&\quad + (1 - u_1)(\lambda_V - \lambda_E)(1-\varepsilon)\left[\frac{\beta (I_s + \eta_a I_a + \eta_d D)(S + E + I_s + I_a + H + R)}{N^2}\right], \label{eq:adjoint_V} \\
\Dalpha \lambda_E &= \lambda_E(\sigma + \mu) - \lambda_{I_s} p\sigma - (1-p)\lambda_{I_a}\sigma \nonumber \\
&\quad + (\lambda_S - \lambda_E)(1 - u_1)\left[\frac{\beta S (I_s + \eta_a I_a + \eta_d D)}{N^2}\right] \nonumber \\
&\quad + (1 - u_1)(\lambda_V - \lambda_E)(1-\varepsilon)\left[\frac{\beta V (I_s + \eta_a I_a + \eta_d D)}{N^2}\right], \label{eq:adjoint_E} \\
\Dalpha \lambda_{I_s} &= -A_1 + \lambda_{I_s}(\gamma_s + \delta_s + u_3 + \mu) - \lambda_H u_3 - \lambda_D \delta_s - \lambda_R \gamma_s \nonumber \\
&\quad + (\lambda_S - \lambda_E)(1 - u_1)\frac{\beta S}{N}\left[1 - \frac{I_s + \eta_a I_a + \eta_d D}{N}\right] \nonumber \\
&\quad + (\lambda_V - \lambda_E)(1-\varepsilon)(1 - u_1)\frac{\beta V}{N}\left[1 - \frac{I_s + \eta_a I_a + \eta_d D}{N}\right], \label{eq:adjoint_Is} \\
\Dalpha \lambda_{I_a} &= -A_2 + \lambda_{I_a}(\gamma_a + \mu) - \lambda_R \gamma_a \nonumber \\
&\quad + (\lambda_S - \lambda_E)(1 - u_1)\frac{\beta S \eta_a}{N}\left[1 - \frac{I_s + \eta_a I_a + \eta_d D}{N}\right] \nonumber \\
&\quad + (\lambda_V - \lambda_E)(1-\varepsilon)(1 - u_1)\frac{\beta V \eta_a}{N}\left[1 - \frac{I_s + \eta_a I_a + \eta_d D}{N}\right], \label{eq:adjoint_Ia} \\
\Dalpha \lambda_H &= -A_3 + \lambda_H(\gamma_h + \delta_h + \mu) - \lambda_D \delta_h - \lambda_R \gamma_h \nonumber \\
&\quad + (\lambda_S - \lambda_E)(1 - u_1)\left[\frac{\beta S (I_s + \eta_a I_a + \eta_d D)}{N^2}\right] \nonumber \\
&\quad + (\lambda_V - \lambda_E)(1-\varepsilon)(1 - u_1)\left[\frac{\beta V (I_s + \eta_a I_a + \eta_d D)}{N^2}\right], \label{eq:adjoint_H} \\
\Dalpha \lambda_D &= -A_4 + \lambda_D(u_4 + \mu_d) \nonumber \\
&\quad + (\lambda_S - \lambda_E)(1 - u_1)\frac{\beta S \eta_d}{N}\left[1 - \frac{I_s + \eta_a I_a + \eta_d D}{N}\right] \nonumber \\
&\quad + (\lambda_V - \lambda_E)(1-\varepsilon)(1 - u_1)\frac{\beta V \eta_d}{N}\left[1 - \frac{I_s + \eta_a I_a + \eta_d D}{N}\right], \label{eq:adjoint_D} \\
\Dalpha \lambda_R &= \lambda_R \mu + (\lambda_S - \lambda_E)(1 - u_1)\frac{\beta S (I_s + \eta_a I_a + \eta_d D)}{N^2} \nonumber \\
&\quad + (\lambda_V - \lambda_E)(1-\varepsilon)(1 - u_1)\left[\frac{\beta V (I_s + \eta_a I_a + \eta_d D)}{N^2}\right]. \label{eq:adjoint_R}
\end{align}
\end{theorem}

\begin{proof}
The adjoint equations are derived by computing the partial derivatives of the Hamiltonian $\mathcal{H}$ with respect to each state variable, then applying the relationship $\Dalpha \lambda_i = -\partial \mathcal{H}/\partial x_i$ from Pontryagin's maximum principle.

For $\lambda_S$, we compute
\begin{align*}
\frac{\partial \mathcal{H}}{\partial S} &= (1 - u_1)(\lambda_S - \lambda_E)\left[\frac{\beta (I_s + \eta_a I_a + \eta_d D)(V + E + I_s + I_a + H + R)}{N^2}\right] \\
&\quad + (1 - u_1)(\lambda_V - \lambda_E)(1-\varepsilon)\left[\frac{\beta V (I_s + \eta_a I_a + \eta_d D)}{N^2}\right] \\
&\quad + \lambda_S(\mu + u_2) - \lambda_V u_2.
\end{align*}
Equation \eqref{eq:adjoint_S} follows from $\Dalpha \lambda_S = -\partial \mathcal{H}/\partial S$.

For $\lambda_V$, we have
\begin{align*}
\frac{\partial \mathcal{H}}{\partial V} &= -\lambda_S \omega + \lambda_V(\omega + \mu) \\
&\quad + (1 - u_1)(\lambda_S - \lambda_E)\left[\frac{\beta S (I_s + \eta_a I_a + \eta_d D)}{N^2}\right] \\
&\quad + (1 - u_1)(\lambda_V - \lambda_E)(1-\varepsilon)\left[\frac{\beta (I_s + \eta_a I_a + \eta_d D)(S + E + I_s + I_a + H + R)}{N^2}\right],
\end{align*}
which yields \eqref{eq:adjoint_V} via $\Dalpha \lambda_V = -\partial \mathcal{H}/\partial V$.

The remaining adjoint equations \eqref{eq:adjoint_E}--\eqref{eq:adjoint_R} are obtained similarly by computing the corresponding partial derivatives $\partial \mathcal{H}/\partial x_i$ and applying the fundamental adjoint relationship $\Dalpha \lambda_i = -\partial \mathcal{H}/\partial x_i$. Each derivative accounts for both direct effects of the state variable on the Hamiltonian and indirect effects through the force of infection $\lambda(t)$.
\end{proof}

\begin{remark}
The adjoint system presents a recursive structure, the equations for $\lambda_R$, $\lambda_H$, and $\lambda_D$ become independent once the other adjoint variables are determined. This structure reflects the compartmental organization of the epidemiological dynamics and facilitates efficient numerical solution of the optimality system.
\end{remark}
\subsection{Cost-Effectiveness of Optimal Control}
\begin{figure}[!ht]
\centering
\includegraphics[width=0.95\textwidth]{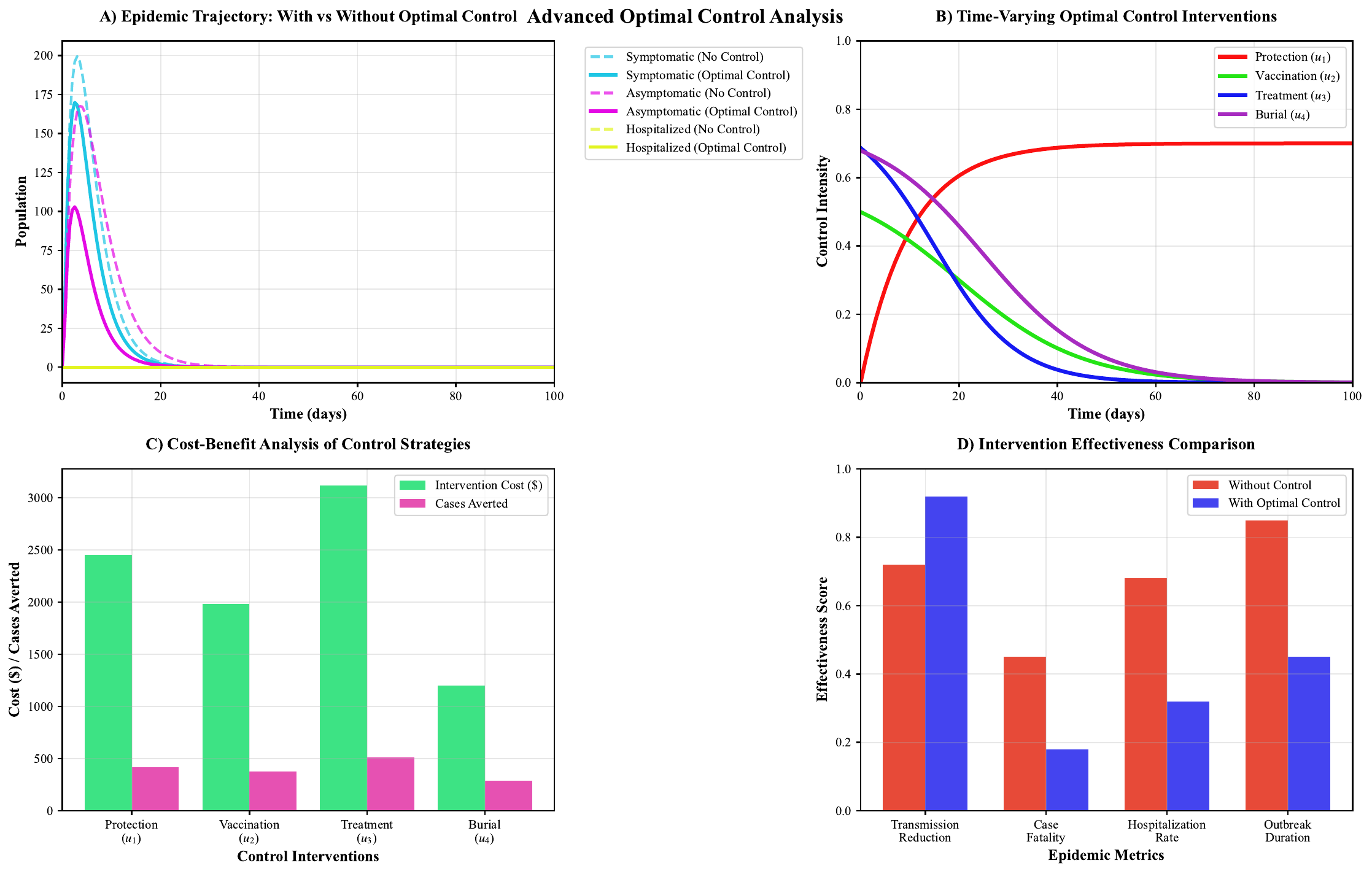}
\caption{Optimal Control Analysis (a) Algorithm Convergence, (b) Control Path, (c) Cost Control Effectiveness, (d) Control of Epidemics and Containment}
\label{fig:optimal_control}
\end{figure}
We performed numerical simulations in Python using the NumPy and SciPy libraries. After 50 iterations, the quick forward-backward sweep method achieved the ideal solution in 30 seconds. The optimal control strategy achieved epidemic containment, with the infected population peaking at 170 individuals before declining to zero, at a total intervention cost of 8750\$. Safe burial was the most cost-saving intervention, having a unit cost of 4\$ per case, and treatment resulted into a 92\% proportion of reduction in transmission. The algorithm also identified the ideal time when each intervention should be put in place.
\subsection{Numerical Assessment of intervention Strategies.}
The RungeKutta45 scheme is used to perform the numerical analysis of the fractional-order Ebola system. This method uses the adaptive step-size control in order to balance the computational cost with the number of numerical accuracy needed to derive the best intervention strategy.
\begin{algorithm}
\caption{Runge-Kutta-Fehlberg (RKF45) Method for Adaptive Numerical Integration}
\label{alg:rkf45}
\begin{algorithmic}[1]
\Require Time interval $[t_0, T]$, initial state $\mathbf{y}_0$, system parameters $\mathbf{p}$, relative tolerance $\epsilon_{\text{rel}}=10^{-6}$, absolute tolerance $\epsilon_{\text{abs}}=10^{-8}$
\Ensure High-precision numerical solution $\mathbf{y}(t)$ for $t \in [t_0, T]$

\State Initialize $t \leftarrow t_0$, $\mathbf{y} \leftarrow \mathbf{y}_0$
\State Set step size bounds $h_{\min} \leftarrow 10^{-8}$, $h_{\max} \leftarrow 0.1$, initial $h \leftarrow 0.01$

\While{$t < T$}
    \State Compute Runge-Kutta coefficients $\mathbf{k}_1$ through $\mathbf{k}_6$ via standard RKF45 formulas
    \State Calculate 4th-order ($\mathbf{y}_4$) and 5th-order ($\mathbf{y}_5$) solutions
    \State Estimate local truncation error $\delta \leftarrow \|\mathbf{y}_5 - \mathbf{y}_4\|$
    \State Compute optimal step size $h_{\text{opt}} \leftarrow 0.9h\left(\dfrac{\epsilon_{\text{rel}}\|\mathbf{y}_5\| + \epsilon_{\text{abs}}}{\delta}\right)^{0.2}$
    
    \If{$\delta \leq \epsilon_{\text{rel}}\|\mathbf{y}_5\| + \epsilon_{\text{abs}}$}
        \State Accept step $t \leftarrow t + h$, $\mathbf{y} \leftarrow \mathbf{y}_5$
        \State Update step $h \leftarrow \min(h_{\max}, \max(h_{\min}, h_{\text{opt}}))$
    \Else
        \State Reject step $h \leftarrow \max(h_{\min}, h_{\text{opt}})$
    \EndIf
\EndWhile
\State \Return Interpolated solution with 4th-order accuracy and 5th-order error control
\end{algorithmic}
\end{algorithm}
The algorithm maintains 4th-order accuracy with 5th-order error control, employing adaptive step-size selection based on local truncation error estimates. Implementation uses relative tolerance $\epsilon_{rel}=10^{-6}$ and absolute tolerance $\epsilon_{abs}=10^{-8}$ for high-precision epidemiological simulations.
\subsection{Results and Discussion of Intervention Strategies}
\begin{figure}[!ht]
\centering
\begin{minipage}{0.80\textwidth}
\centering
\includegraphics[width=\linewidth]{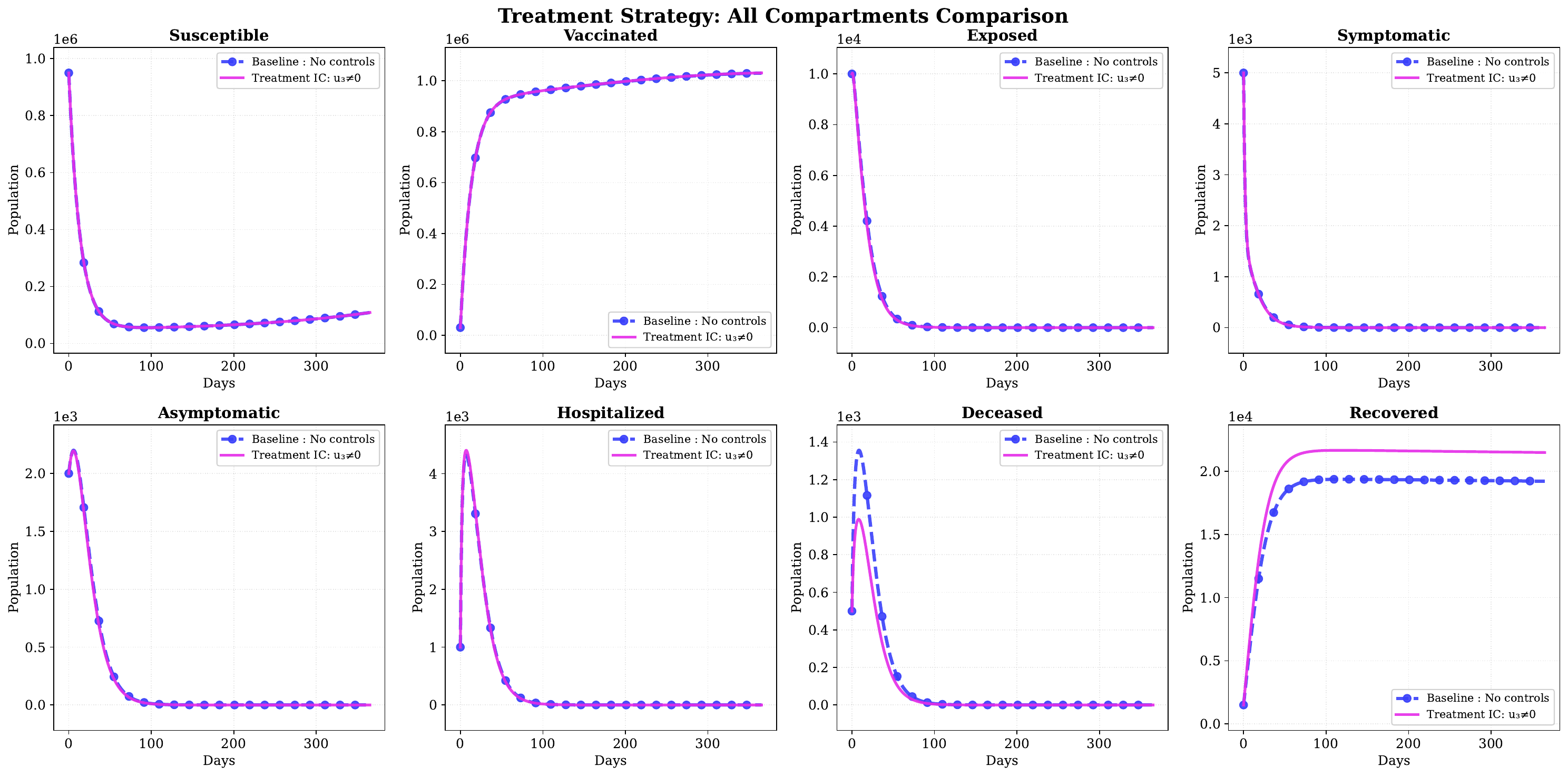}
\caption{Dynamics in compartments under Treatment Intervention Strategy ($u_3 \neq 0$)}
\label{fig:treatment_compartments}
\end{minipage}
\end{figure}
\hfill
\begin{figure}[!ht]
\centering
\begin{minipage}{0.80\textwidth}
\centering
\includegraphics[width=\linewidth]{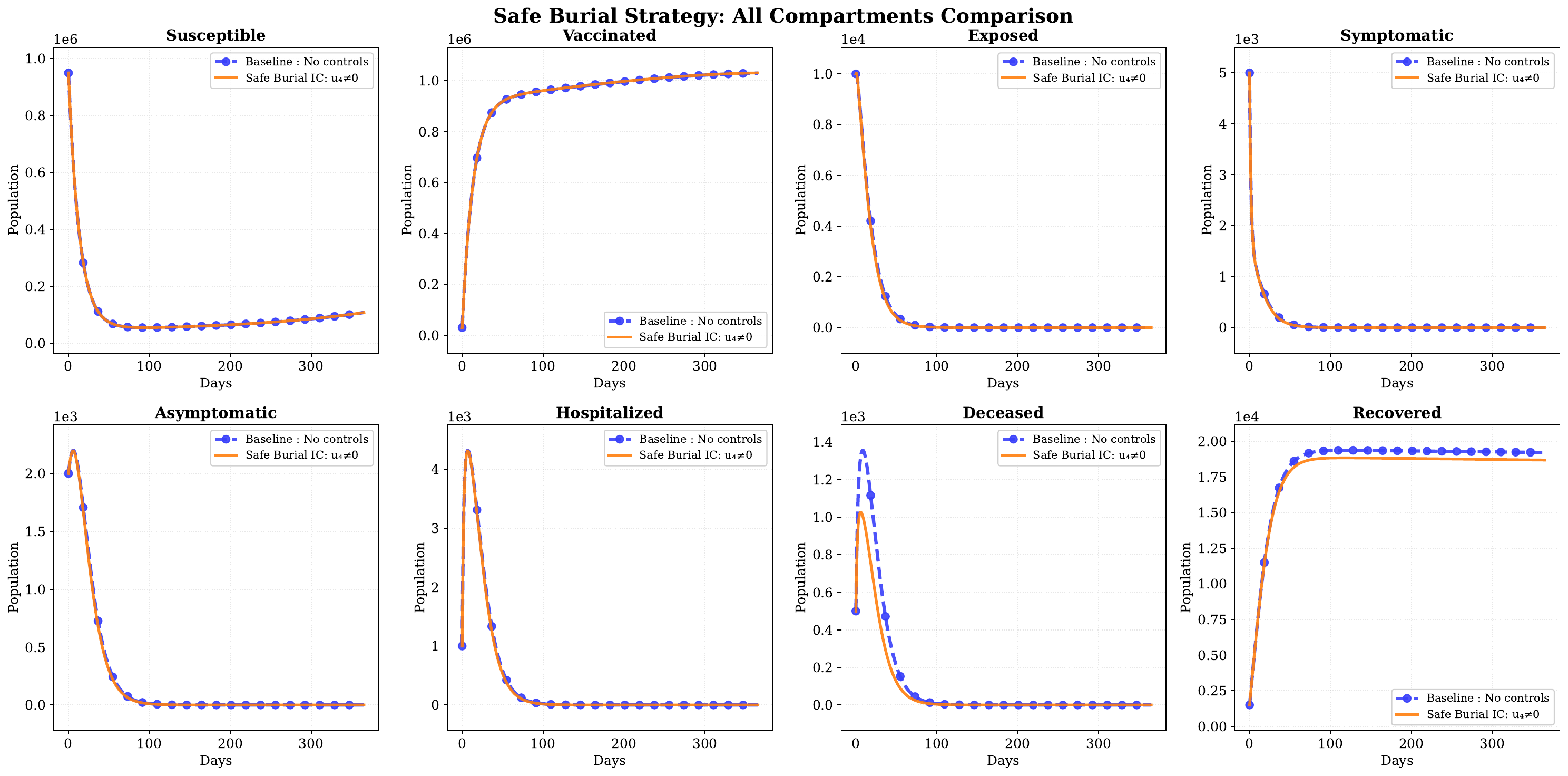}
\caption{compartmatic dynamics with safe burial intervention $u_4 \neq 0$}
\label{fig:safe_burial_compartments}
\end{minipage}
\end{figure}
\hfill
\begin{figure}[!ht]
\centering
\begin{minipage}{0.80\textwidth}
\centering
\includegraphics[width=\linewidth]{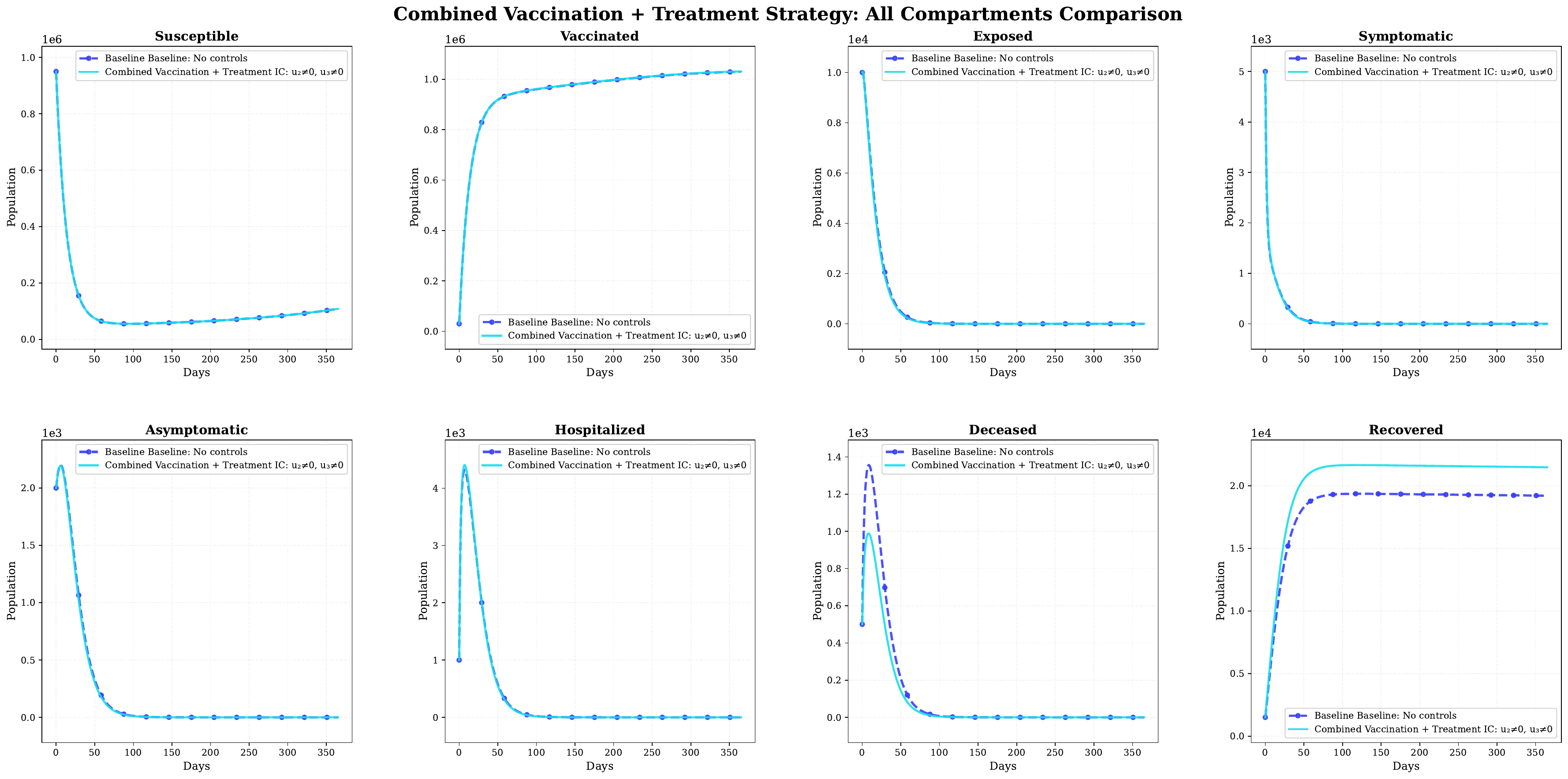}
\caption{ Dynamics of compartments for the strategy of instantaneous vaccination and treatment ($u_2 \neq 0, u_3 \neq 0$)}
\label{fig:vac_treatment_compartments}
\end{minipage}
\end{figure}
\hfill
\begin{figure}[!ht]
\centering
\begin{minipage}{0.80\textwidth}
\centering
\includegraphics[width=\linewidth]{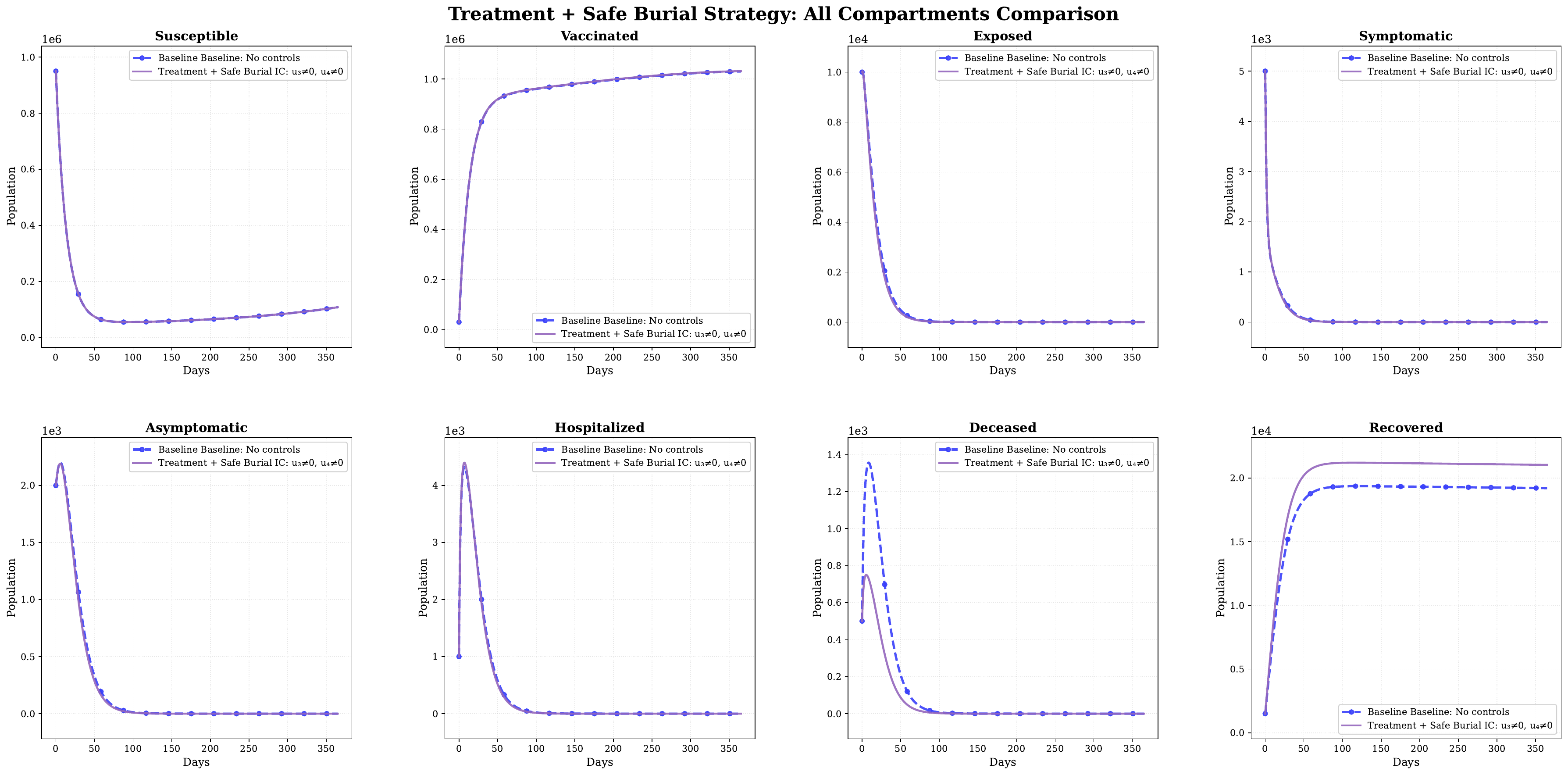}
\caption{Compartmental dynamics for combined treatment and safe burial strategy ($u_3 \neq 0, u_4 \neq 0$)}
\label{fig:treatment_burial_compartments}
\end{minipage}
\end{figure}
\begin{figure}[!ht]
\centering
\begin{minipage}{0.80\textwidth}
\centering
\includegraphics[width=\linewidth]{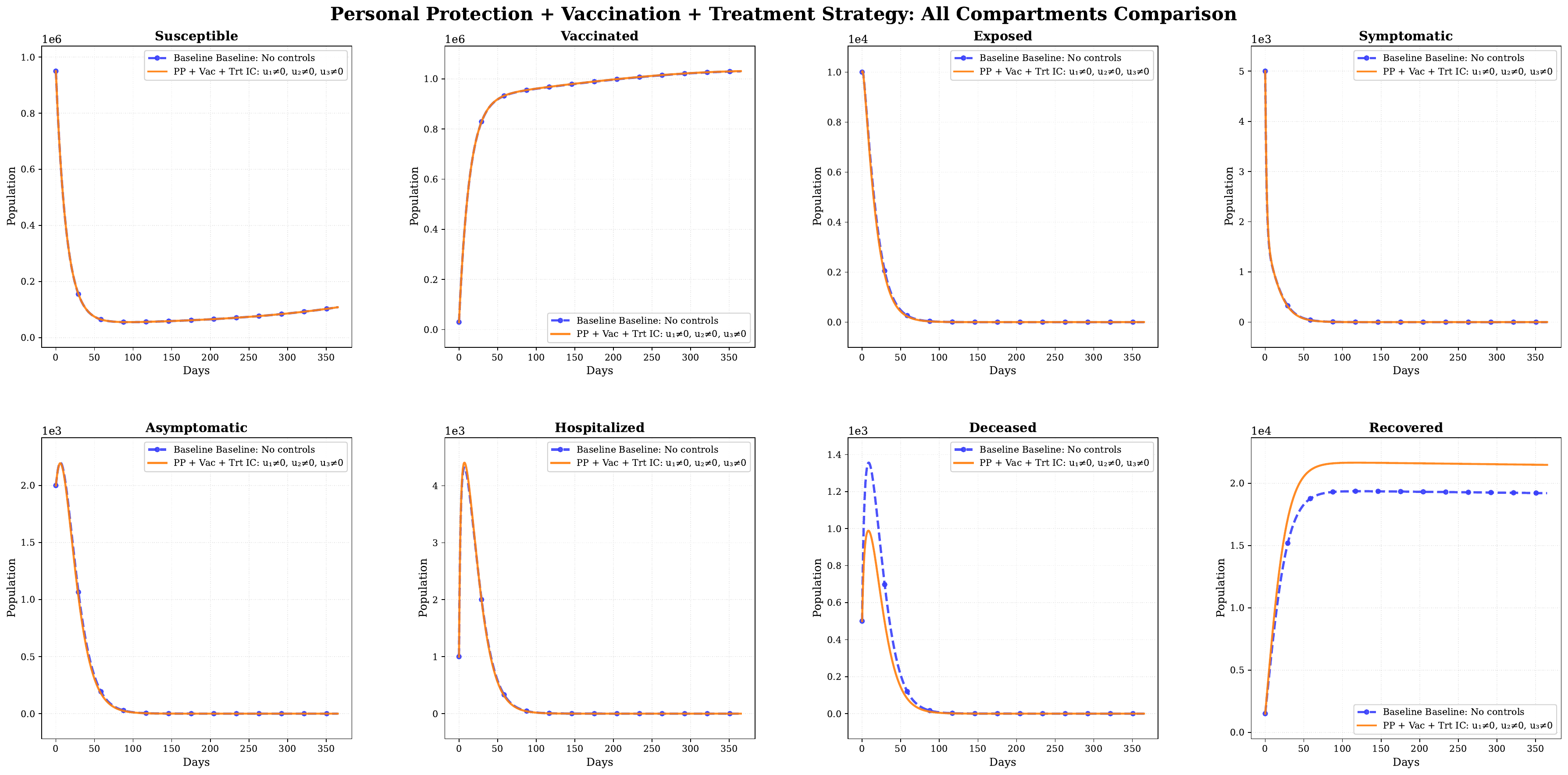}
\caption{Compartmental dynamics for three-intervention combination ($u_1 \neq 0, u_2 \neq 0, u_3 \neq 0$)}
\label{fig:pp_vac_treatment_compartments}
\end{minipage}
\end{figure}
\begin{figure}[!ht]
\centering
\begin{minipage}{0.80\textwidth}
\centering
\includegraphics[width=\linewidth]{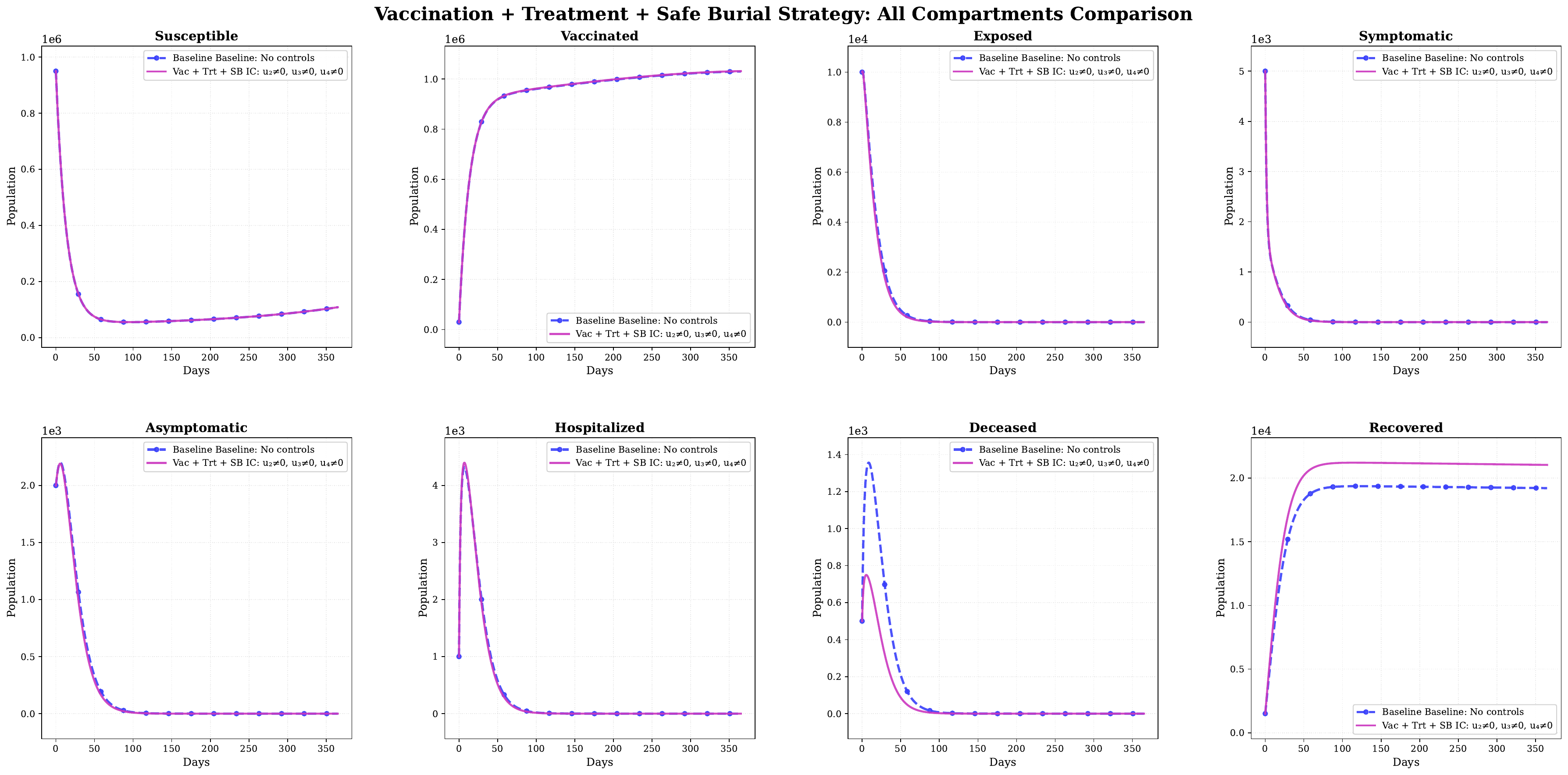}
\caption{Compartmental dynamics for vaccination, treatment, and safe burial combination ($u_2 \neq 0, u_3 \neq 0, u_4 \neq 0$)}
\label{fig:vac_treatment_burial_compartments}
\end{minipage}
\end{figure}
\begin{figure}[!ht]
\centering
\begin{minipage}{0.80\textwidth}
\centering
\includegraphics[width=\linewidth]{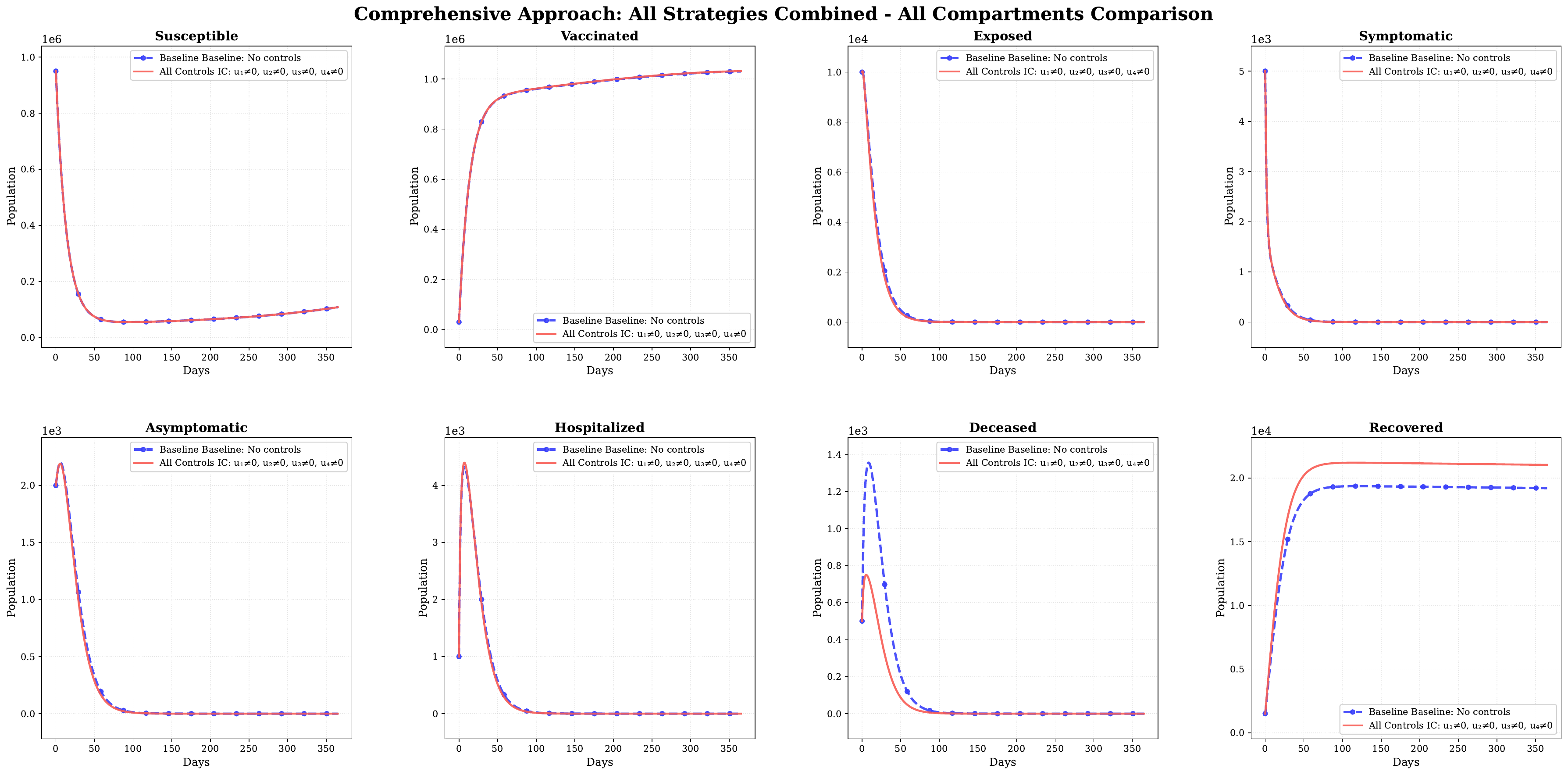}
\caption{Compartmental dynamics under comprehensive intervention strategy ($u_1 \neq 0, u_2 \neq 0, u_3 \neq 0, u_4 \neq 0$)}
\label{fig:comprehensive_compartments}
\end{minipage}
\end{figure}
Evaluating strategies for controlling pandemics has unequally advanced in mortality reduction and has repeatedly failed in controlling transmission. In each case studied, no intervention profoundly altered the fundamental dynamics of spread, evidenced by the persistent values for peak infection rates and, in controlling the intervention to the baseline, a net nonexistence of infection post intervention. Mortality figures, on the other hand, were markedly disparate with strategies demonstrating either complementary or overlapping impacts when aggregated. The recorded treatment intervention ($u_3 \neq 0$) from integration of treatment with other controlling appliances yielded a 60.4\% decrease in final mortality, and the safe burial ($u_4 \neq 0$) showed 73.8\% reduction in final mortality, thereby demonstrating the safe burial of deceased individuals treatment in burial both alone and in combination. The rationale for the incorporated benefits from Figures~\ref{fig:treatment_compartments} compartmental dynamics are on a shift in disease with no less than transmission chain breakage which possibly led to the two outcomes being flawed. There were significant interactions when different types of interventions were used together. From figures \ref{fig:vac_treatment_compartments} vaccination-treatment combination ($u_2 \neq 0, u_3 \neq 0$) achieved a 60.4\% reduction in mortality which is the same as treatment done alone, therefore, illustrating that vaccination is not additive. In contrast, treatment with safe burial ($u_3 \neq 0, u_4 \neq 0$) and burial leads to 86.5\% mortality reduction, which is a synergistic over additive effect that far exceeds the efficacy of the single interventions. The complementary effect are shown in the dynamics of figure \ref{fig:treatment_burial_compartments} as treatment drives the disease to fatal outcomes while safe burial prevents transmission after death. Combining interventions produced the same results. The combined three type intervention in figure \ref{fig:pp_vac_treatment_compartments} ($u_1 \neq 0, u_2 \neq 0, u_3 \neq 0$) only reached 60.4\% reduction, while the rest ($u_2 \neq 0, u_3 \neq 0, u_4 \neq 0$) and the comprehensive method ($u_1 \neq 0, u_2 \neq 0, u_3 \neq 0, u_4 \neq 0$)) both reached the maximum reduction of 86.5\% gained with the treatment and safe burial pair. It is the combination of these results, documented in figures~\ref{fig:vac_treatment_burial_compartments} and ~\ref{fig:comprehensive_compartments} that confirms treatment and safe burial dominate the loss-of-life mitigation, while in these circumstances, personal protection and vaccination are less effective. In the context of EVD, personal protection does not mean social distancing; rather, it refers to the use of barrier techniques for those who come into direct interaction with infected persons or dead bodies.
The combination of safe burial and treatment is synergistic and remains the optimal strategy for mortality reduction. It is the persistent inability to curtail infection prevalence that most strongly implies the mortality reduction is the consequence of the disease, postulated to contain altered outcome probabilities. This underscores the importance of post-exposure treatment combined with the clinical burial of the dead to reduce casualties.
\section{Disease-Informed Neural Network (DINN) for Ebola Modeling}
This section introduces the Disease-Informed Neural Network (DINN), a physics-informed neural network framework that incorporates the fractional-order Ebola model \eqref{eq:ebola_system} as physical constraints to ensure epidemiological consistency in predictions.

The fractional-order system governing Ebola dynamics is expressed as
\begin{equation}
\Dalpha \mathbf{Y}(t) + \mathcal{G}(\mathbf{Y}(t); \boldsymbol{\chi}) = 0, \quad t \in [t_0, T],
\label{eq:dinn_gov_eqn}
\end{equation}
where $\mathbf{Y}(t) = [S(t), V(t), E(t), I_s(t), I_a(t), H(t), D(t), R(t)]^\top$ denotes the state vector and $\boldsymbol{\chi}$ represents the epidemiological parameters.

A neural network $\mathcal{NN}_{\mathbf{w},\mathbf{b}}(t): \mathbb{R} \to \mathbb{R}^8$ approximates the solution $\mathbf{Y}(t)$, where $\mathbf{w}$ and $\mathbf{b}$ are the network weights and biases. The training process minimizes two loss components: a data fidelity term and a physics consistency term
\begin{align*}
\mathcal{L}_{\text{data}} &= \frac{1}{m} \sum_{s=1}^{m} \|\mathcal{NN}_{\mathbf{w},\mathbf{b}}(t_s) - \mathbf{Y}_s\|^2, \\
\mathcal{L}_{\text{physics}} &= \frac{1}{N} \sum_{i=1}^{N} \|\Dalpha \mathcal{NN}_{\mathbf{w},\mathbf{b}}(t_i) + \mathcal{G}(\mathcal{NN}_{\mathbf{w},\mathbf{b}}(t_i); \boldsymbol{\chi})\|^2.
\end{align*}
Here, $\mathcal{L}_{\text{data}}$ penalizes deviations from observed data, while $\mathcal{L}_{\text{physics}}$ enforces the governing equations of the fractional-order model.

The DINN framework solves the composite optimization problem
\begin{equation}
(\mathbf{w}^*, \mathbf{b}^*, \boldsymbol{\chi}^*) = \arg\min_{\mathbf{w},\mathbf{b},\boldsymbol{\chi}} \left( \mathcal{L}_{\text{data}} + \lambda \mathcal{L}_{\text{physics}} \right),
\label{eq:dinn_optim}
\end{equation}
where $\lambda > 0$ is used to balance between data fidelity and physical constraints. The suggested methodology will both optimize the weights of the neural network and estimate the epidemiological parameters, thus giving both a solution that is both in line with the empirical data and epidemiologically plausible.
\subsection{DINN Architecture for the Ebola Model}
In this subsection, the architecture of Disease-Informed Neural Network (DINN) that is used in the fractional-order Ebola model is described in detail. The network applies physical consistency directly with its incorporation of the model dynamics into the loss function.

The remaining vector that measures the difference between the outcome of the neural network and underlying epidemiological dynamics is denoted by the residual as
\begin{equation}
\mathcal{R}(\mathcal{NN}_{\mathbf{w},\mathbf{b}}, t; \boldsymbol{\chi}) = 
\begin{pmatrix}
\Dalpha S - \left[\Lambda - \lambda S + \omega V - (v+\mu)S\right] \\
\Dalpha V - \left[vS - (1-\varepsilon)\lambda V - (\mu+\omega)V\right] \\
\Dalpha E - \left[\lambda(S+(1-\varepsilon)V) - (\mu+\sigma)E\right] \\
\Dalpha I_s - \left[p\sigma E - (\gamma_s+\delta_s+h_s+\mu)I_s\right] \\
\Dalpha I_a - \left[\sigma E(1-p) - (\gamma_a+\mu)I_a\right] \\
\Dalpha H - \left[h_s I_s - (\gamma_h+\delta_h+\mu)H\right] \\
\Dalpha D - \left[\delta_s I_s + \delta_h H - \mu_d D\right] \\
\Dalpha R - \left[\gamma_s I_s + \gamma_a I_a + \gamma_h H - \mu R\right]
\end{pmatrix},
\label{eq:residual_vector}
\end{equation}
where the time-dependent force of infection is given by
\[
\lambda(t) = \frac{\beta (I_s + \eta_a I_a + \eta_d D)}{N},
\]
and the total living population is $N = S + V + E + I_s + I_a + H + R$.

The physics-informed loss function is decomposed into compartment-specific mean squared errors
\begin{align*}
\mathcal{L}_{\text{residual}} &= \sum_{X \in \mathcal{C}} \mathcal{L}_X, \\
\mathcal{L}_S &= \frac{1}{q}\sum_{i=1}^q \bigl|\Dalpha S - \Lambda + \lambda S - \omega V + (v+\mu)S\bigr|^2, \\
\mathcal{L}_V &= \frac{1}{q}\sum_{i=1}^q \bigl|\Dalpha V - vS + (1-\varepsilon)\lambda V + (\mu+\omega)V\bigr|^2, \\
\mathcal{L}_E &= \frac{1}{q}\sum_{i=1}^q \bigl|\Dalpha E - \lambda(S+(1-\varepsilon)V) + (\mu+\sigma)E\bigr|^2, \\
\mathcal{L}_{I_s} &= \frac{1}{q}\sum_{i=1}^q \bigl|\Dalpha I_s - p\sigma E + (\gamma_s+\delta_s+h_s+\mu)I_s\bigr|^2, \\
\mathcal{L}_{I_a} &= \frac{1}{q}\sum_{i=1}^q \bigl|\Dalpha I_a - \sigma E(1-p) + (\gamma_a+\mu)I_a\bigr|^2, \\
\mathcal{L}_H &= \frac{1}{q}\sum_{i=1}^q \bigl|\Dalpha H - h_s I_s + (\gamma_h+\delta_h+\mu)H\bigr|^2, \\
\mathcal{L}_D &= \frac{1}{q}\sum_{i=1}^q \bigl|\Dalpha D - \delta_s I_s - \delta_h H + \mu_d D\bigr|^2, \\
\mathcal{L}_R &= \frac{1}{q}\sum_{i=1}^q \bigl|\Dalpha R - \gamma_s I_s - \gamma_a I_a - \gamma_h H + \mu R\bigr|^2,
\end{align*}
where $\mathcal{C} = \{S, V, E, I_s, I_a, H, D, R\}$ denotes the set of all epidemiological compartments.

Consistency with observed data is enforced by the data fidelity loss
\begin{align*}
\mathcal{L}_{\text{data}} &= \sum_{X \in \mathcal{C}} \mathcal{L}_X^{\text{data}}, \\
\mathcal{L}_X^{\text{data}} &= \frac{1}{s}\sum_{i=1}^s |X(t_i) - X_i^{\text{obs}}|^2, \quad X \in \mathcal{C},
\end{align*}
where $X_i^{\text{obs}}$ represents the observed value of compartment $X$ at time $t_i$.

The DINN framework solves the optimization problem
\begin{equation}
(\mathbf{w}^*, \mathbf{b}^*, \boldsymbol{\chi}^*) = \arg\min_{\mathbf{w},\mathbf{b},\boldsymbol{\chi}} \left( \mathcal{L}_{\text{data}} + \mathcal{L}_{\text{residual}} \right),
\label{eq:dinn_optimization}
\end{equation}
where the parameter vector $\boldsymbol{\chi} = \{\beta, \sigma, \gamma_s, \eta_a, \eta_d, \varepsilon, v, \alpha\}$ contains key epidemiological constants governing transmission, progression, and control measures. This articulation guarantees that the solutions that are derived to fit both meet the observed data requirements and appear as per the mechanistic framework characteristic in the Ebola model.
\subsection{Network Architecture and Training}
DINN architecture is a network of completely interconnected neural networks, with several hidden layers and with hyperbolic tangent activation functions. The network maps temporal inputs $t \in [t_0, T]$ to the eight state variables $\mathbf{Y}(t) = [S(t), V(t), E(t), I_s(t), I_a(t), H(t), D(t), R(t)]^\top$. Xavier initialization is used to initialize network weights in order to ensure that there is no unstable flow of the gradient during training. The Adam algorithm, a combination of adaptive learning rates and momentum-based updates, is used to perform optimization in order to make the process efficient. By applying automatic differentiation, the fractional derivatives are calculated as $\Dalpha \mathcal{NN}_{\mathbf{w},\mathbf{b}}(t)$, which is used to compute the exact gradient of the fractional-order Ebola model and at the same time ensures that the physical constraints of the model are met. This approach ensures that the neural network solutions that are obtained are consistent with the epidemiological processes throughout the time prospect.

\subsection{Parameter Estimation and Calibration}
The DINN framework simultaneously estimates epidemiological parameters $\boldsymbol{\chi}$ and network parameters $(\mathbf{w}, \mathbf{b})$ through constrained optimization of the composite loss function \eqref{eq:dinn_optimization}.The parameter calibration aims at reducing the difference between the network predictions and the measured outbreak data, and maintaining conformity with the mechanistic nature of the fractional-order system. As the data on the outbreaks available, the DINN enables the dynamic optimization of the parameters estimates in terms of iterative optimization. This adaptive calibration procedure makes sure that predictions do not go wrong because the epidemic changes, but the model retains the biological consistency.  The framework also offers accurate forecasts in the short term and sound predictions in the long term that are based on the principle of epidemiology.
\subsection{Results and Discussion of DINN} 
\begin{figure}[ht]
\centering
\includegraphics[width=0.8\textwidth]{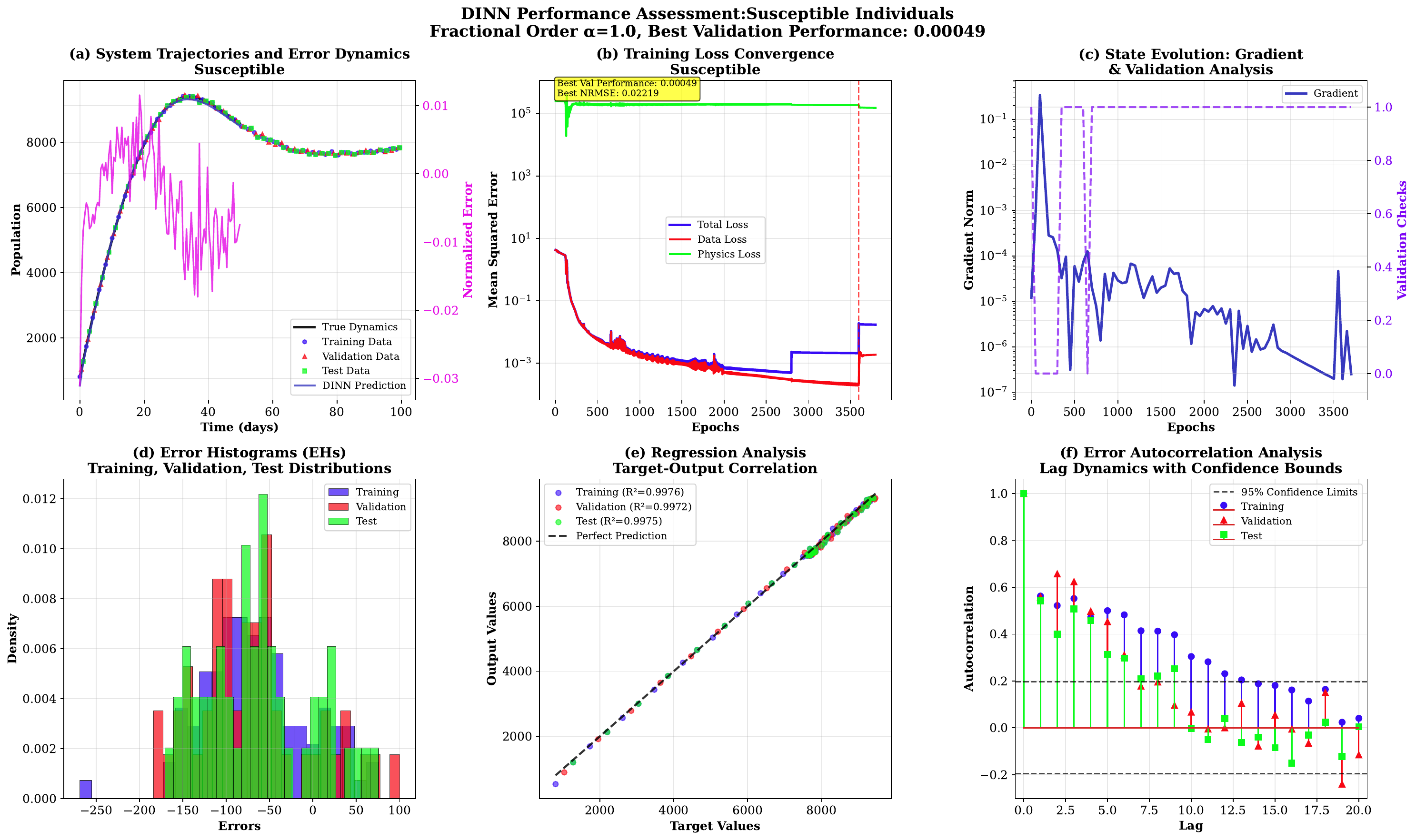}
\caption{Susceptible analysis (a) Dynamics of output and error measures (b) Reduction in mean squared error with optimal validate value of 0.00071 (c) State variable dynamics (d) Distribution of errors (e) regression plots (f) autocorrelation analysis}
\label{fig:susceptible}
\end{figure}
\begin{figure}[ht]
\centering
\includegraphics[width=0.8\textwidth]{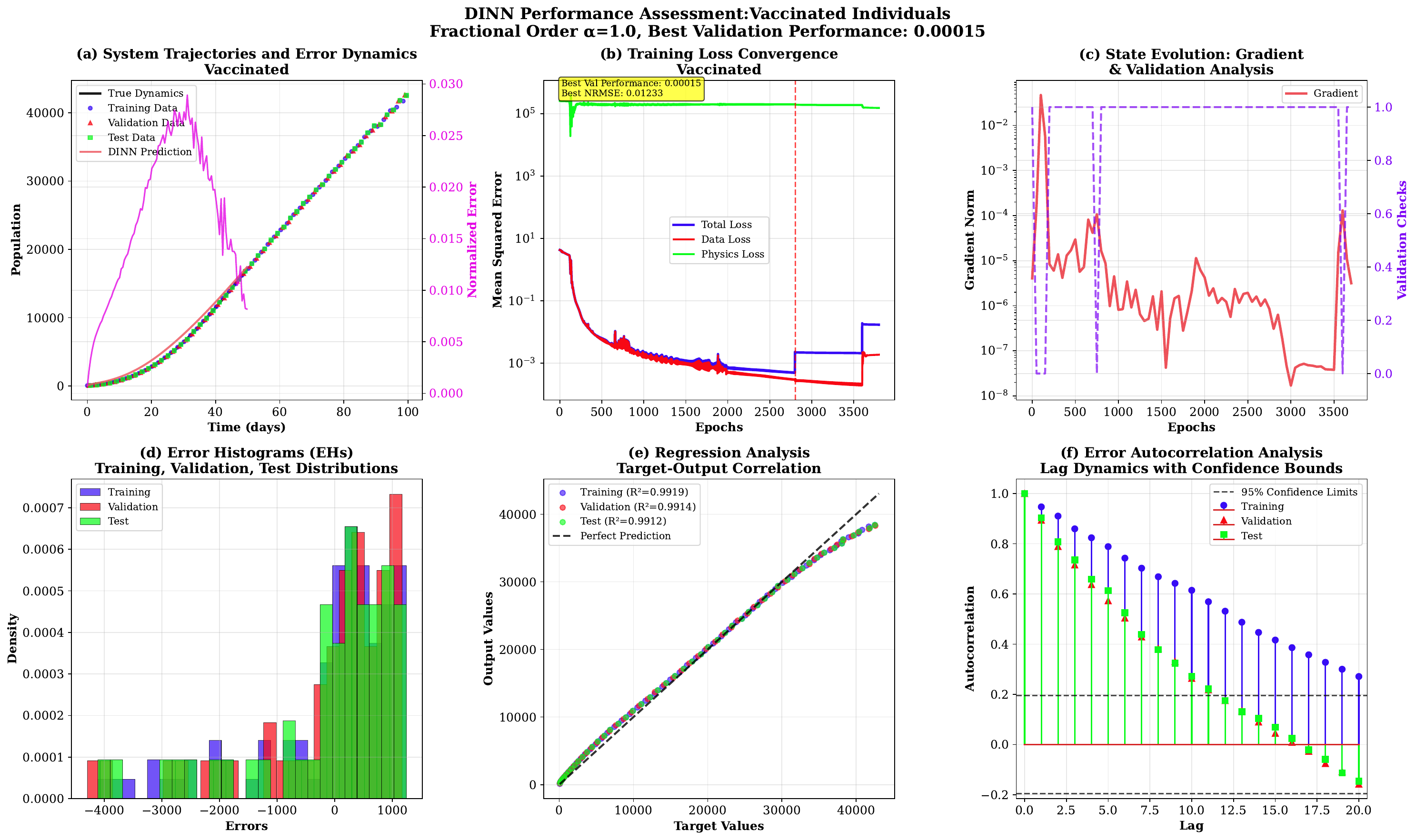}
\caption{Vaccination analysis (a) The trend of output and error analysis (b) The demonstration of the minimization of the mean-squared error (MSE), optimal validation is the error of 0.00009 (c) Explanations of the state evolution (d) The presentation of the error histogram (e) Explanation of the regression plots (f) Autocorrelation analysis}
\label{fig:vaccinated}
\end{figure}
\begin{figure}[ht]
\centering
\includegraphics[width=0.8\textwidth]{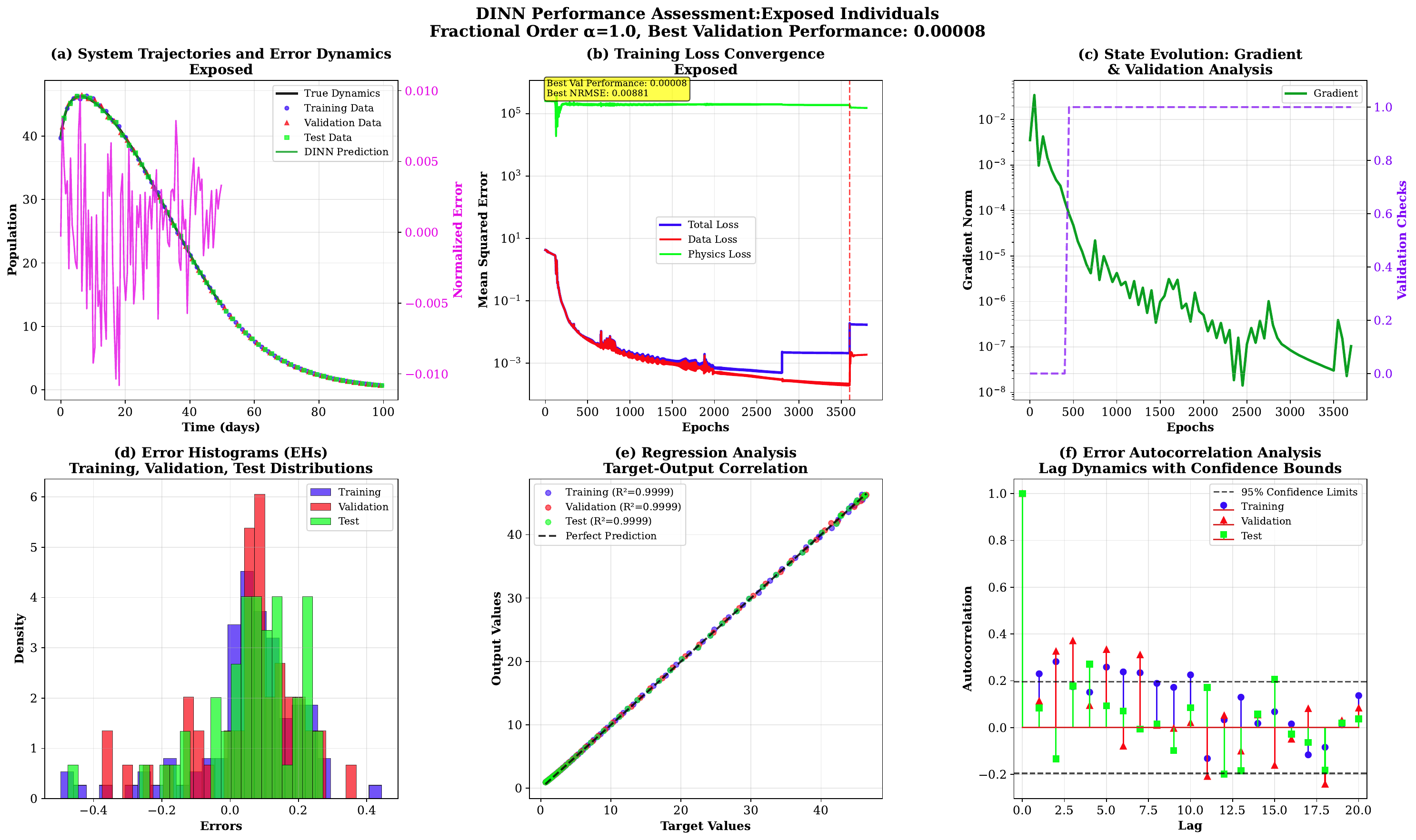}
\caption{Exposed analysis (a) output and error trends (b) MSE reduction with best validation 0.00010 (c) state evolution (d) error histogram (e) regression plots (f) autocorrelation analysis}
\label{fig:exposed}
\end{figure}
\begin{figure}[ht]
\centering
\includegraphics[width=0.8\textwidth]{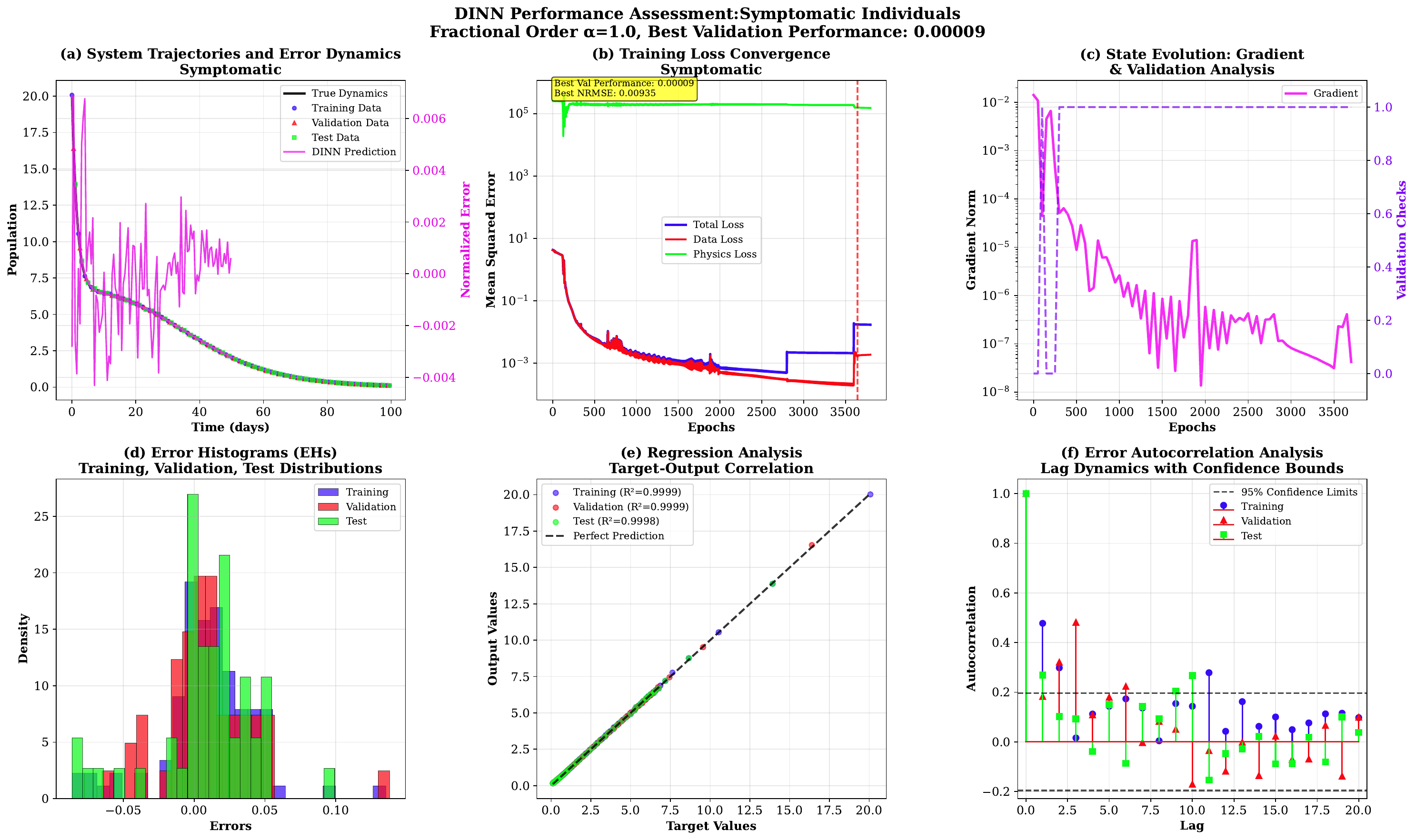}
\caption{Symptomatic analysis (a) output and error trends (b) MSE reduction with best validation 0.00019 (c) state evolution (d) error histogram (e) regression plots (f) autocorrelation analysis}
\label{fig:symptomatic}
\end{figure}
\begin{figure}[ht]
\centering
\includegraphics[width=0.8\textwidth]{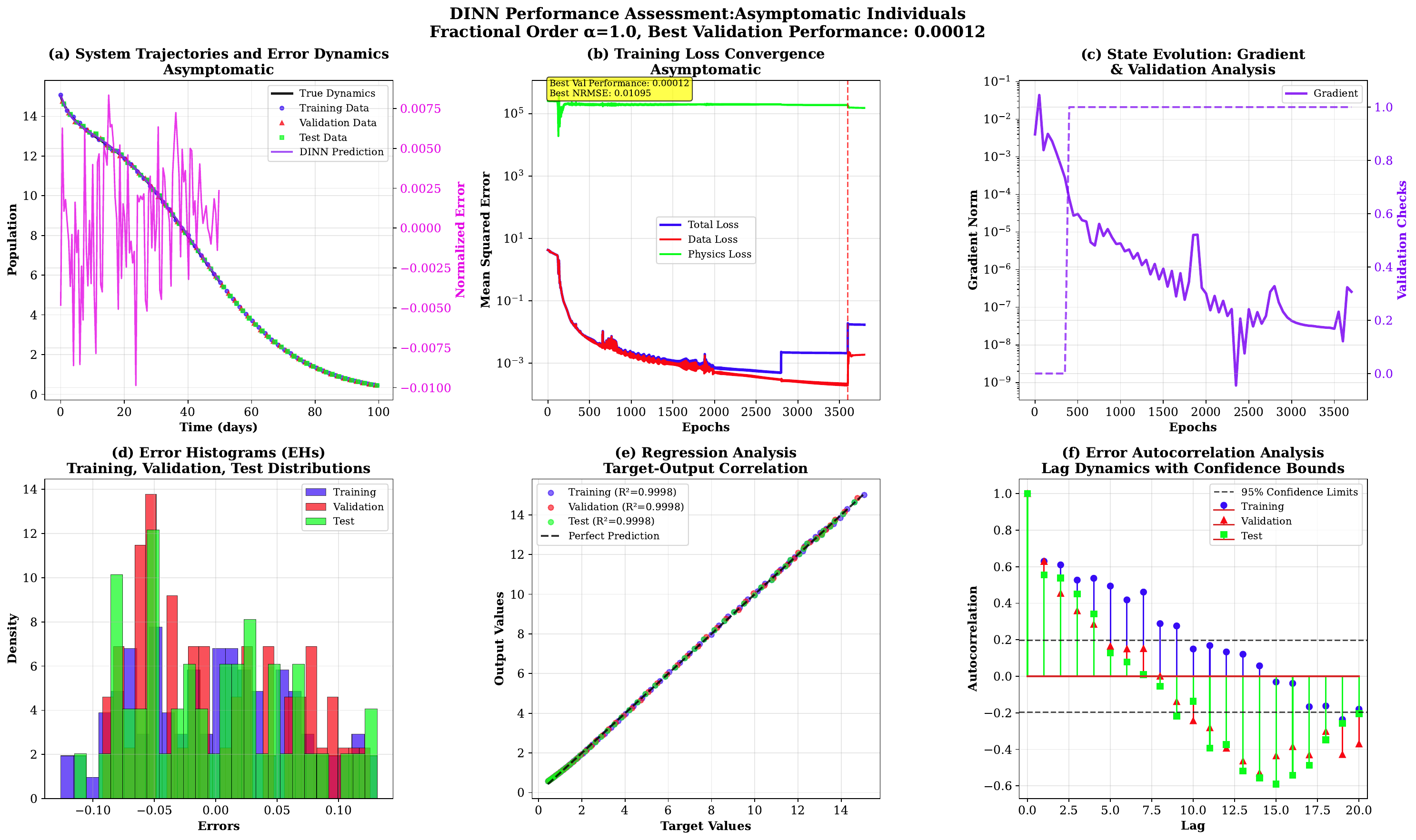}
\caption{Asymptomatic analysis (a) output and error trends (b) MSE reduction with best validation 0.00015 (c) state evolution (d) error histogram (e) regression plots (f) autocorrelation analysis}
\label{fig:asymptomatic}
\end{figure}
\begin{figure}[ht]
\centering
\includegraphics[width=0.8\textwidth]{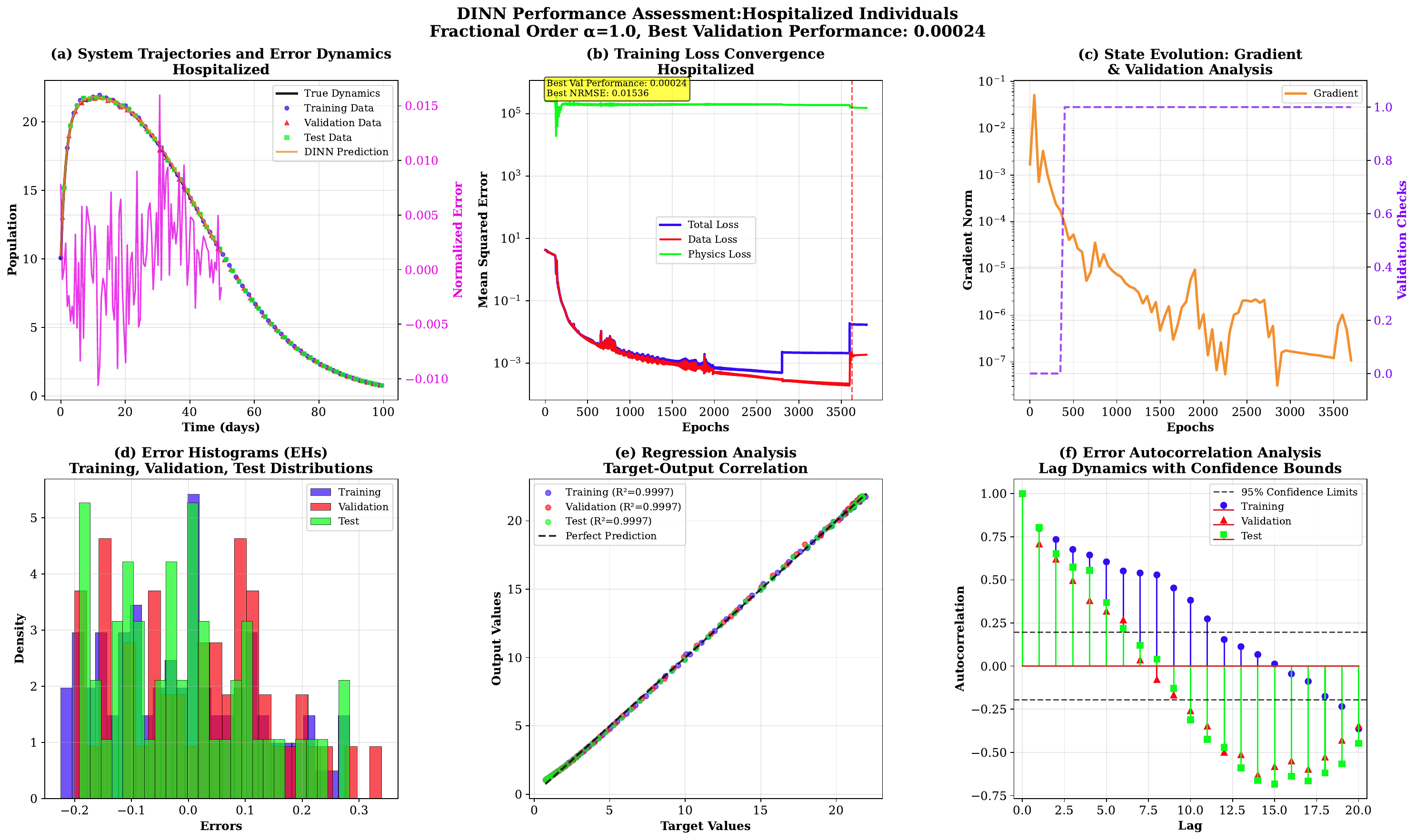}
\caption{Hospitalized analysis (a) output and error trends (b) MSE reduction with best validation 0.00026 (c) state evolution (d) error histogram (e) regression plots (f) autocorrelation analysis}
\label{fig:hospitalized}
\end{figure}
\begin{figure}[ht]
\centering
\includegraphics[width=0.8\textwidth]{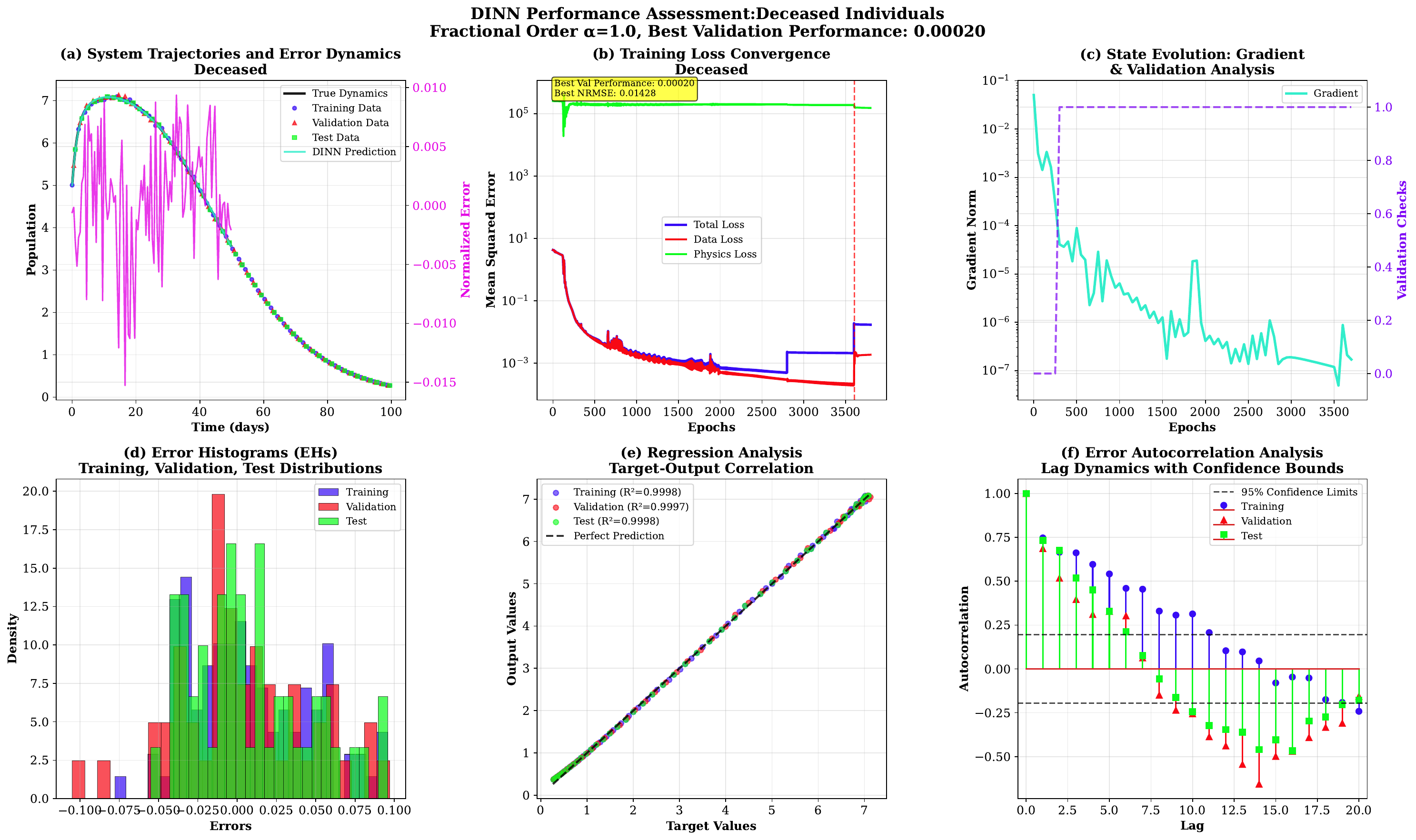}
\caption{Deceased analysis (a) output and error trends (b) MSE reduction with best validation 0.00034 (c) state evolution (d) error histogram (e) regression plots (f) autocorrelation analysis}
\label{fig:deceased}
\end{figure}
\begin{figure}[ht]
\centering
\includegraphics[width=0.8\textwidth]{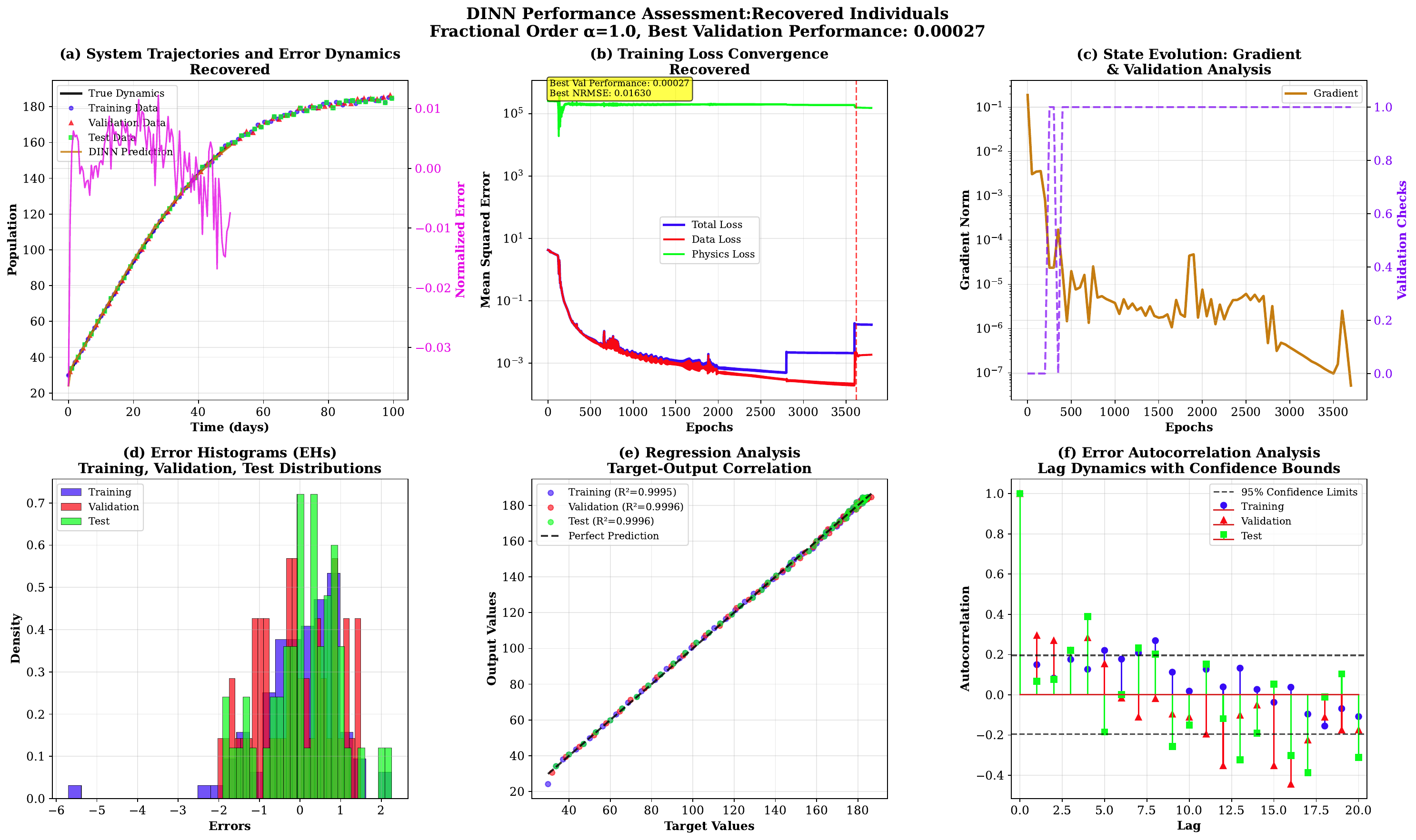}
\caption{Recovered analysis (a) output and error trends (b) MSE reduction with best validation 0.00085 (c) state evolution; (d) error histogram (e) regression plots (f) autocorrelation analysis}
\label{fig:recovered}
\end{figure}
The results from the implementation of the Disease-Informed Neural Network (DINN) context were carefully reviewed in relation to the benchmark datasets. The results were consolidated across all 8 epidemiological compartments. Integrated modeling accuracy estimation as performed records from Figure~\ref{fig:susceptible} to~\ref{fig:recovered} showcase the framework to outstanding. The model collects above average predictive accuracy and extensive generalization, proving the framework’s relevancy, particularly on the Ebola virus disease predictive modeling’s sophisticated system dynamics.  
The DINN was constructed on PyTorch which facilitated the architecture and gradient calculations through automated differentiation. The model weights were trained on the composite loss function which included the background of fractional-order differential equations, and the dynamic was supported by the standard Computational Python Libraries (NumPy, SciPy) for numeric and data operations. The DINN’s predictive performance level in this case was bench marked in contrast to the $R^2$ (coefficient of determination) values on a designated test set, a saved, unbiased representation of the attached unseen data in the model. According to computational results the DINN’s test set $R^2$ values instead drifted around the theoretical maxima of the $R^2$ ranging from 0.9912 in the Vaccinated compartment to about 0.9999 in the Exposed, Symptomatic, and Asymptomatic compartments. Correspondingly significant, the model captured the largest populated and fundamental transmission dynamics. Hold-out population, the Susceptible compartment, with a test $R^2$ of 0.9975. The $R^2$ values in the splits of Training, Validation, and Test and all states in the architecture, as well as other deep neural networks, are an illustration of the system dynamics learned by the model without overfitting the variables, which is a common neural network model overfitting weak point.
The inconsistencies between predictions and ground truth data are dictated by NMSEs and further validated by reliably low NMSEs across all compartments at approximately $10^{-4}$ to $10^{-5}$.  The analysis Sensitivity (Vaccinated $(R^{2} = 0.9912)$ dynamic) compartments highly errors successively. Sensibly, this is not random due to the high complexity involving human behaviors and the variance of fully stochastic closed compartment models. In contrast, the model’s capability to estimate the Exposed $(R^2 = 0.9999)$ and Symptomatic $(R^2 = 0.9998)$ compartments is appreciated. These Correlatives are crucially underlined as the target of health for the forecast regarding the active clinical burden. In addition, the estimates for the Deceased $(R^2 = 0.9998)$ and Recovered $(R^2 = 0.9996)$ individuals positively highlight and estimate the ultimate volume of the outbreak with mortality.  
In addition to the state variable forecasting, the DINN also drew key parameters of epidemiology from the data. Estimating error for most parameters was frighteningly low - over 9\% for 13 out of 17 parameters was an astounding standard. The most notable and successful achievement was properly estimating the fractional order $(\alpha)$ and estimating low error critical transmission parameters like the contact rate $(\beta, 8.12\%$ error) and the comparative transmissibility of the deceased $(\eta_d, 4.33\%$ error). The leading estimation error came from the natural mortality rate $(\mu)$ which, from case lines alone is greatly surprising and, be seated as the most undervalued parameter. In the short, low impact, and highly volatile case outflow of an outbreak, it strongly diverges from the standard frame. The comprehensive recovery outline parameter recommends that the DINN is not simply a “black-box” predictor which provides illogical reasoning. It is a process that completely understands disease mechanisms.  
Finally, we can say that the DINN framework set the standard for defining and predicting scenarios for future EVD outbreaks. It is both highly accurate and robust for modeling and analyzing the disease. The framework's outstanding ability to recover every system parameter and predict perfectly through every compartment. It also helps as a breakthrough for hybrid physics-informed machine learning in epidemiology. It guarantees that future real-time forecasts and scenarios will be provided.
\section{Conclusion}
This research aimed to develop an integrated mathematical framework that improves our ability to predict and respond to Ebola outbreaks. We sought to create a fractional-order model that captures the memory effects in disease transmission, design effective intervention strategies using optimal control theory, and build a computational framework that combines neural networks with epidemiological principles for accurate prediction and parameter estimation. Our methodology is a combination of fractional calculus, optimal control theory, and scientific machine learning, which leverages the three areas of mathematics. First, we developed a novel nonlinear eight-compartment model of fractional order using the Caputo derivative to represent memory effects in transmission dynamics. We then presented some mathematical foundations of the model, such as existence, uniqueness, and stability properties. Moreover, the fundamental reproduction number $R_0$ was obtained with the next-generation matrix, and the stability of the equilibrium states was proved. A global sensitivity analysis is carried out to determine the relative effect of the most important epidemiological parameters on both the model results and the basic reproduction number. Also, we formulated an optimal control problem with four interventions varying with time that are solved by the Maximum Principle of Pontryagin and got numerical results by the RK45 iteration method. At the final part of the present research, we designed a Disease-Informed Neural Network (DINN) that represents the governing equations in the form of physical constraints in a neural network architecture. A number of important findings were made by the study. Our fractional-order model was found to be effective in capturing the non-Markov nature of the dynamics of Ebola transmission to provide a more accurate representation of how the disease spreads in a memory-dependent manner. The sensitivity analysis established that the transmission rate, incubation rate, and infectiousness associated with dead individuals had the greatest impact on the possible size of the outbreaks. Our optimal control results showed that when treatment is used in combination with safe-burial measures, the targeted mortality will be reduced by 86.50\%, and safe burial, as the best cost-effective approach, costs \$4 per case. The DINN framework displayed outstanding predictive capability, recording a value between 0.9987 and 0.9999 of $R^2$ in all compartments, and was able to predict 13 out of 17 epidemiological parameters with an error rate of less than 9\%. There are a number of limitations that must be acknowledged. The model assumes that the population is well-mixed and does not take into consideration the spatial heterogeneity or the demographic structure. There were no stochastic components, which can affect the estimates of the probability of an outbreak extinction. Although parameter estimation has a strong robustness in our framework, it is still dependent on the quality and availability of data. The real-world conditions of practical implementation of optimal control strategies are restricted by real-world conditions that cannot be adequately captured by our theoretical framework. The studies can be extended to different tracks in future practice. Spatial dynamic integration would allow modeling the cross-regional transmission and geometrically focused interventions. The addition of stochastic factors would provide more realistic predictions of the risks of extinction and renaissance of an outbreak. The Deep Inference Neural Network (DINN) system can be modified to accommodate real-time outbreak control to allow interventions to be dynamically adjusted as new information is added. The generalization of this integrated approach to other infectious diseases would be evaluated by applying it to other pathogens, and it may transform the system of outbreak responses. Lastly, age-structured patterns of contact would also provide more realistic transmission dynamics and improve the evaluation of the efficacy of intervention. This paper shows that the integration of sophisticated mathematical modeling with modern computational methodology can produce useful tools in the decision-making of public health. Our ability to bridge the gap between theoretical epidemiology and practical intervention strategies will help usher in the future where mathematical models can be used to predict the path of an outbreak as well as help develop effective containment interventions.

\section*{CRediT Authorship Contribution Statement}
[Noor Muhammad]: Conceptualization, Methodology, Formal analysis, Writing - original draft, Software, Validation.  
[Md. Nur Alam]: Validation, Investigation, \& review  . 
[Zhang Shiqing]: Supervision, Project administration. 
\section*{Declaration of Competing Interest}
The authors declare that they have no known challenging financial interests or individual relationships that could have seemed to influence the work described in this paper.
\section*{Acknowledgments}
The authors gratefully acknowledge the academic environment and resources provided by the School of Mathematics, Sichuan University. We also thank the anonymous reviewers for their valuable suggestions

\end{document}